\documentclass{amsart}
\usepackage[utf8]{inputenc}
\usepackage{microtype}
\usepackage{amssymb,amsmath,amsthm}

\newtheorem{theorem}{Theorem}[section]
\newtheorem{lemma}[theorem]{Lemma}
\newtheorem{corollary}[theorem]{Corollary}
\newtheorem{claim}[theorem]{Claim}
\newtheorem{proposition}[theorem]{Proposition}
\newtheorem{example}[theorem]{Example}

\numberwithin{equation}{section}

\DeclareMathOperator{\SL}{SL}
\DeclareMathOperator{\SO}{SO}

\newcommand{\nil}{\text{Nil}}
\DeclareMathOperator{\GL}{GL}
\DeclareMathOperator{\id}{Id}

\newcommand{\V}{\mathcal{V}}
\newcommand{\R}{\mathbb{R}}
\newcommand{\Z}{\mathbb{Z}}

\newcommand{\E}{\mathbb{E}}
\newcommand{\PP}{\mathbb{P}}
\newcommand{\F}{\mathcal{F}}

\newcommand{\X}{\mathcal{X}}
\newcommand{\Y}{\mathcal{Y}}

\newcommand{\om}{\omega}
\newcommand{\s}{\sigma}

\newcommand{\de}{\delta}
\newcommand{\g}{\gamma}
\newcommand{\la}{\lambda}

\newcommand{\Om}{\Omega}

\newcommand{\tF}{\widetilde{\F}}
\newcommand{\tnu}{\widetilde{\nu}}
\newcommand{\tmu}{\widetilde{\mu}}
\newcommand{\tE}{\widetilde{E}}

\newcommand{\onefiner}{\overset{1}{\prec}}

\newcommand{\ov}{\overline}
\newcommand{\un}{\underline}

\begin{document}
\title{Exact dimension of dynamical stationary  measures}
\author{Fran\c cois Ledrappier and Pablo Lessa}
\address{Fran\c cois Ledrappier, Universit\'e de Paris et Sorbonne Universit\'e, CNRS, LPSM, Bo\^{i}te Courrier 158, 4, Place Jussieu, 75252 PARIS cedex
05, France,} \email{fledrapp@nd.edu}
\address{ Pablo Lessa,  IMERL, Facultad de Ingeniería, Julio Herrera y Reissig 565,  11300 Montevideo, Uruguay}  \email{plessa@fing.edu.uy}

\thanks{FL was partially supported by IFUM; PL thanks CSIC research project 389}

 \maketitle
 \begin{abstract} We consider a random walk on \(\SL_d(\R) \) with finite first moment and countable support. We show that the distributions of the unstable flag space 
and of the stable flag space are exact dimensional. \end{abstract}

 \section{Introduction}
 
 \subsection{Exact dimension}
 
 We are given a probability measure \(\mu\) on the group \(\GL_d(\R)\) and we are interested in the asymptotic properties of products of sequences of matrices chosen independently with distribution \(\mu.\) The logarithm of the determinant follows a classical law of large numbers and we may restrict ourselves to probability measures on \(\SL_d(\R) \). 
 We assume that the probability \(\mu \) has a  finite first moment. Then, the behaviour of individual vectors under the random product of matrices is described by Oseledets theorem (\cite{oseledets}, see section \ref{section:oseledets}). Namely, there are numbers, the Lyapunov exponents \( \chi _1 > \ldots > \chi _N ,\) with multiplicities \( d_1, \ldots, d_N ,\) and, at almost every \(\om\),  a random decomposition \[ \R^d \; = \; E_1(\om) \oplus \ldots \oplus E_N(\om) \] with \( \dim E_i (\om ) = d_i\) for \( i = 1 \ldots, N\) such that \[ v \not = 0 \in E_i(\om) \; \iff \; \lim\limits _{n \to \pm \infty } \frac {1}{n} \log \| g^{(n)}(\om) v \| = \chi _i.\] The space \(\X\) of such decompositions is an open  subset of \( \Pi _{j=1}^N {\textrm {Gr}}_{d_i}(d) \) and is canonically endowed with a Borel structure and a Riemannian metric. From our results follows that the distribution of the Oseledets decomposition in \(\X\) is exact-dimensional. We have 
 \begin{theorem}\label{Mexact} Assume that the measure \(\mu\)  has countable support. Let \(M\) be the distribution on \(\X\) of the Oseledets decomposition: for any Borel \(A \subset \X, \; M(A) \) is   the probability that  \( \left(  E_1(\om),\ldots , E_N(\om) \right)\) belongs to \(A.\)  Then, there is a number \( \Delta \) such that, for \(M\)-a.e. \(x \in \X\):
 \[ \lim \limits_{r\to 0} \frac{\log M(B(x,r))}{\log r} \; = \; \Delta .\] \end{theorem}
 
 The space \(\X\) can be seen as a  space of pairs of  partial flags in general position in \(\R^d.\) With that setting,  the measure \(M\) is the  product of the natural dynamical stationary measure \( \nu \) on the space of unstable partial flags \(\F\) and the natural dynamical stationary measure \( \nu' \) on the space of stable partial flags \( \F'\). The space \( (\F, \nu )\) is the Furstenberg boundary of the random walk \( (\SL_d(\R), \mu).\) We show (theorem \ref{finiteentropytheorem}) that the measure \(\nu\) is stationary ergodic and that the Furstenberg entropy \(  \kappa(\mu,\nu) \) (see section \ref{section:entropy}) satisfies
  \begin{equation}\label{majorationentropy}\kappa(\mu,\nu)\; \le \; \sum\limits_{0<i<j\leq N}d_i d_j(\chi_i - \chi_j).\end{equation} 
  
  Assume that the measure \(\mu\)  has countable support. We show (theorem \ref{exactdimensiontheorem}) that the measure \(\nu \) is exact dimensional with dimension \(\de\). Since the measure \(\nu '\) is the dynamical stationary measure for the random walk \( (\SL_d(\R) , \mu' ),\) where \( \mu'(g):= \mu (g^{-1}),\) the measure \( \nu '\) is exact dimensional as well (with dimension \(\de' \)).
  Theorem \ref{Mexact} follows since the measure \(M\) is the product of the measures \(\nu \) and \( \nu'.\) Moreover, \(\Delta = \de + \de'.\) 
  
  \subsection{Structure of the flag space}
  
  Exact dimension holds for any coarser partial flag space, and theorem \ref{exactdimensiontheorem} is stated for a general partial flag space. The case of the projective space \(\PP^{d-1}\) is due to A. Rapaport who extended previous work for \(\SL_2(\R)\) and affine iterated function systems (see \cite{feng-hu}, \cite{hochman-solomyak}, \cite{barany-kaenmaki},  \cite{feng} and \cite{rapaport}). The dimension of \(\nu _{\PP^{d-1}} \) is given in \cite{rapaport} by a sum \( \sum _{j=1}^{N-1} \g _j ,\) where \( \g _j \) are partial dimensions. In the general case of the dynamical stationary measure on the unstable flag space, the dimension \( \de\) is  a sum of partial dimensions \( \g _{i,j}\), for \( 0<i <j \le N.\)
  Fix the multiplicities \( d_1, \ldots, d_N \) such that \( d_1 + \ldots + d_N =d \) and consider the corresponding spaces \(\X, \F \) and \( \F'.\) The main novelty in this paper is the description of the bundle \( \X \to \F' \) as an array  of \(\SL_d(\R) \)-equivariant algebraic finite-dimensional vector bundle extensions (see theorems  \ref{vectorbundletheorem} and \ref{changeofcoordinatestheorem}). To each vector extension is associated a pair \( (i_k,j_k), 0<i_k<j_k \le N\) and the linear dimension of the fiber space is \( d_{i_k}d_{j_k}.\)

  Given \( \mu \) on \( \SL_d(\R ) \) with finite first moment and multiplicities  \( d_1, \ldots, d_N ,\) for each vector extension, there exists a family of conditional measures on the fibers associated to \(M\). These conditional measures are \(M\)-almost everywhere  defined. They are \(M\)-a.e. exact dimensional (theorem \ref{onesteptheorem}) and the  dimension \(\g_k\) 
 is \(M\)-a.e. constant.  We have \begin{equation}\label{majorationdimension} \g _k \; \le \;  d_{i_k}d_{j_k}. \end{equation} 
 Moreover, for each vector  extension, there is a dynamically defined entropy \( \kappa _k \) that satisfies \( \kappa _k = \g_k (\chi _{i_k} -\chi _{j_k})\)   (theorem \ref{onesteptheorem}).
   In \cite{ll1}, we assume that the multiplicities \(d_i\) are all $1$, so that the bundle structure is simpler to describe and the proof of theorem \ref{onesteptheorem} is more direct.
   
 \subsection{Addition formulas}
 The structure of the array of extensions is not linear (see Figures 1 and 2 in \cite{ll1} for the cases \(N= 3\) and \(N=4\)) but there are paths \( k_1, \ldots , k_{N(N-1)/2} \) of successive extensions from \(\F'\) to  \( \X \). Moreover, the associated pairs \( i_{k_\ell}, j_{k_\ell} \) are such that all pairs \( (i,j), \, 0< i <j \le N \) appear and appear once. Along any such path, we have the entropy formula (see section \ref{entropy}): \begin{equation}\label{entropyadd} \kappa (\mu, \nu ) \;= \; \sum _{\ell = 1}^{N(N-1)/2} \kappa _{k_\ell} \;= \; \sum _{\ell = 1}^{N(N-1)/2} \g _{k_\ell} (\chi _{i_{k_\ell}} -\chi _{j_{k_\ell}}) . \end{equation}
 
 We obtain the exact dimension of the measure \(\nu \) in theorem \ref{exactdimensiontheorem} by adding the dimensions along a path of successive extensions from \(\F'\) to  \(M \) under two more conditions:
 \begin{enumerate}
 \item we consider a {\it{monotone }} path, such that  \( \ell \mapsto  \chi _{i_{k_\ell}} - \chi _{ j_{k_\ell}} \) is non-decreasing;
 \item we assume that the measure \(\mu \)  has countable support.
  \end{enumerate}
  Under those two conditions, the dimensions add (see theorem \ref{additivitytheorem}):
  \begin{equation}\label{dimensionadd}  \de  \;= \; \sum _{\ell = 1}^{N(N-1)/2} \g _{k_\ell} .\end{equation}
  
  The first condition is necessary: there are counter-examples already in \(\SL_3(\R)\), see \cite{ll1}, section 9. It is an open question whether the second condition can be weakened. %The proofs of formulas (\ref{entropyadd}) and (\ref{dimensionadd}) in this paper are similar to those in \cite{ll1}. 
 The conditions for the formula  (\ref{dimensionadd}) is the main difference with \cite{feng} and \cite{rapaport}, where they consider the action of the projective space. In that case, there is only one path of successive extensions, and it is automatically monotone. We also replace the exponential moment condition with finite first moment and countable support.

 \subsection{Lyapunov dimension}
 
 We continue assuming that  the measure \(\mu \)  has countable support and a first moment. Following \cite{kaplan-yorke} and \cite{douady-oesterle}, we define the Lyapunov dimension of a stationary measure on \(\F\) and show the relation (\ref{Lyadim}). 
 
  \
  
 We order the differences \( \la _k = \chi _{i_k} - \chi _{j_k} \) in such a way that \[0<  \la _1 \le \la _2 \le \ldots \le \la _{N(N-1)/2} .\] 
 Define the continuous, piecewise affine  function \( D_{\F ,\mu ,\nu} \) on the interval \( [ 0,\dim \F]\)
 as:
\begin{eqnarray*} D_{\F, \mu ,\nu } (0) &=& \kappa( \mu, \nu) \;{\textrm { and }}\\  D'_{\F, \mu ,\nu} (s) &=& -\la_\ell \; {\textrm {for}} \;\sum _{k < \ell }d_{i_k}d_{j_k} <s <\sum _{k \le \ell }d_{i_k}d_{j_k}, \; \ell = 1, \ldots, N(N-1)/2. \end{eqnarray*}
We define the  {\it {Lyapunov dimension}} \( \dim_{{\textrm{LY}}}(\F, \mu ,\nu )\) as the number  
 such that \[ D_{\F,\mu,\nu} ( \dim_{{\textrm{LY}}}(\F, \mu ,\nu ) ) = 0.\] 
 By (\ref{majorationentropy}), the definition makes sense: the function \(D_{\F,\mu, \nu} \) is decreasing and  \[ D_{\F, \mu,\nu } (\dim \F) = \kappa (\mu, \nu ) - \sum\limits_{0<i<j\leq N}d_i d_j(\chi_i - \chi_j) \le 0.\]
 \begin{theorem} Assume that  the measure \(\mu \) on \(\SL_d(\R) \)  has countable support and a first moment. Let \(\nu \) be the dynamical stationary measure on the space of unstable flags. With the above notations
 \begin{equation}\label{Lyadim} \de = \dim \nu \; \le  \; \dim_{{\textrm{LY}}}(\F, \mu ,\nu  ) .\end{equation} \end{theorem}
\begin{proof} By  (\ref{majorationdimension}), \(  \g _k \; \le \;  d_{i_k}d_{j_k}\). Since we are following  a monotone path, we can apply (\ref{dimensionadd}) and write, since the slope of  the function \( D_{\F, \mu ,\nu} \) is at least \(\chi_{j_\ell} - \chi_{i_\ell}\) for \( \sum _{k < \ell }\g _k <s <\sum _{k \le \ell }\g _k, \)
 \[ D_{\F, \mu ,\nu } (\de ) = D_{\F, \mu ,\nu } (\sum _\ell \g _\ell) \geq \kappa (\mu, \nu ) - \sum\limits_{k}\g _\ell(\chi_{i_\ell} - \chi_{j_\ell}) .\]
By (\ref{entropyadd}), we have \( D_{\F,\mu ,\nu} (\de) \ge 0 .\) \end{proof}

From the proof, we see that if we have equality in (\ref{Lyadim}), then the \( \g _\ell \) are known in terms of the entropy \( \kappa (\mu, \nu ) \): there is some \(k_0\) such that 
\begin{itemize}
\item \( \g _\ell \;=\; d_{i_\ell} d_{j_\ell } \) for \( \ell < k_0 ,\) 
\item \( \g _\ell \;=\; 0\) for \( \ell > k_0 \)  and 
\item \( \g_{k_0} \) is such that  (\ref{entropyadd}) holds.
\end{itemize}
In other words,  we have equality in (\ref{Lyadim}) if the measure \(\nu \) aligns along the directions with the less contraction, the only constraint being the Furstenberg entropy. An open problem is to find conditions under which we have equality in (\ref{Lyadim}). This is true when \(d=2\) and in some examples in higher dimensions, see  \cite{feng-simon} for the latest results and a history of similar problems, in particular for IFS.

 Define the function \(\ov D_{\mu } \) by replacing \( \kappa (\mu, \nu ) \) by the {\it {random walk entropy}} \(h_{\textrm {RW}} (\mu ) := \inf _n \frac{1}{n} H(\mu ^{\ast n} ).\) The function \(\ov D_{\mu } \) does not depend anymore on \(\F , \nu\), but there is no guarantee that   \(\ov D_{\mu }(\dim \F ) \le 0. \) So, for any \(C^1 \; \SL_d(\R) \)-space \(\Y\), we can define \( \ov \dim _\mu (\Y) \) by 
 \begin{eqnarray*} \ov \dim _\mu (\Y) &:= & \dim \Y  {\textrm{ if }} \ov D_\mu (\dim \Y ) \ge 0,\\
    &:= & s {\textrm{ such that  }} \ov D_\mu (s ) = 0 \; {\textrm { otherwise .}}
    \end{eqnarray*}
    
    Since for any stationary measure \(\nu\) on any \(\SL_d(\R)\)-space, \(\kappa (\mu, \nu ) \le h_{\textrm {RW}} (\mu ) \)     (\cite{kaimanovich-vershik}), we have  \( \dim_{{\textrm{LY}}}(\F, \mu ,\nu )\le \ov \dim_\mu (\F )\) and the equality \( \de = \ov \dim _\mu (\F) \) is harder to obtain than equality in (\ref{Lyadim}). See \cite{hochman-solomyak} for the \( d=2 \) case and \cite{rapaport2} and the references therein for the latest developments.
    
    \subsection{Organization of the paper}
    
 In section \ref{results}, we fix the notations and state the precise and complete form of our results.   The proof of theorem  \ref{finiteentropytheorem} is given in section  \ref{theorem:finiteentropytheorem}. We describe the vector bundle structure of \( \X \to \F'\) and the partial vector bundles in section \ref{bundlestructure}, and their algebraic change of coordinates in section \ref{algebraic}. We then describe some  dynamical properties of these bundles in section \ref{lineardynamics}. In section \ref{entropy}, we recall the properties of the partial entropies associated to all the equivariant spaces we constructed. We prove theorem \ref{onesteptheorem} in section \ref{onestep} and theorem \ref{additivitytheorem} in section \ref{final}.

 \section{Statement of results}\label{results}
 
 We fix \(G = \SL_d(\R)\) and a probability \(\mu\) on \(G\).

 We always assume in what follows that \(\mu\) has finite first moment, meaning that \(\int\limits_{G} \log \|g\| d\mu(g) < +\infty\).

 We say \(\mu\) has countable support if it is a countable convex combination of point masses.

 We consider on the sequence space \(\Omega = G^{\Z}\), the Bernoulli measure \(m = \mu^\Z\), we let \(g_n: \Omega \to G\) be the \(n\)-th coordinate projection, and \(\sigma:\Omega \to \Omega\) be the left shift so that \(g_n\circ \sigma = g_{n+1}\).
 We will also use the notation
 \[g_{m}^n(\omega) = g_{n-1}(\omega)\cdots g_{m}(\omega),\]
 for all \(m < n\).
 %\subsection{The dynamical stationary measures}
 
 \subsection{Oseledets splitting}\label{section:oseledets}

By the Oseledets theorem there exists \(N \in \lbrace 1,\ldots, d\rbrace\), multiplicities \(d_1,\ldots,d_{N}\) with \(d_1+\cdots +d_{N} = d\),  and Lyapunov exponents \(\chi_1 > \cdots > \chi_{N}\) such that for \(m\)-a.e. \(\omega\) the set
 \begin{align*}
   E_i(\omega) &= \left\lbrace v \in \R^d: \lim\limits_{n \to +\infty}\frac{1}{n}\log \|g_{0}^{n}(\omega)v\| \le \chi_i\right\rbrace 
            \\ &\cap \left\lbrace v \in \R^d: \lim\limits_{n \to +\infty} \frac{1}{n}\log \|g_{-n}^{0}(\omega)^{-1}v\| \le -\chi_i \right\rbrace,
 \end{align*}
is a \(d_i\)-dimensional subspace of \(\R^d\) for \(i = 1,\ldots,N\). Moreover, for any pair \( I,J\) of disjoint non-empty subsets of \( \{1, \ldots , N\}\)
\begin{equation}\label{Oseledets} \lim\limits _{n \to \pm \infty } \frac {1}{|n|} \log |\sin \angle ( \oplus _{i \in I} E_i (\s^n \om) ,  \oplus _{j \in J} E_j(\s^n \om) )| \; =\; 0,\end{equation}
where \( \angle (E,E') \) is the smallest angle between unit vectors of the vector spaces \(E\) and \(E'\).

\subsection{Entropy and dimension on flag spaces}

\subsubsection{Left filtrations}

For any \(I \subset \lbrace 1,\ldots, N\rbrace\), set \(d(I) = \sum\limits_{i \in I}d_i\). We will call a set 
\[L \subset \lbrace \emptyset, \lbrace 1\rbrace,\lbrace 1,2\rbrace, \ldots, \lbrace 1,\ldots, N\rbrace\rbrace,\]
containing \(\emptyset\) and \(\lbrace 1,\ldots,N\rbrace\) a {\it{ left filtration.}}

For each \(i \in \lbrace 1,\ldots, N\rbrace\) we denote by \(L(i)\) its atom in \(L\), i.e. the smallest element of \(L\) containing  \(i\).

\subsubsection{Flag spaces}

Recall that a flag in \(\R^d\) is a sequence of increasing subspaces beginning with the subspace \(\lbrace 0\rbrace\) and ending in the subspace \(\R^d\).

Given a left filtration \(L\) let \(\F_L\) be the space of partial flags given as sequences \((x_I)_{I \in L}\) such that 
\begin{enumerate}
  \item \(x_I\) is a \(d(I)\)-dimensional subspace of \(\R^d\) for each \(I \in L\) and,
  \item \(x_I \subset x_J\) whenever \(I \subset J\).
\end{enumerate}

Each flag space \(\F_L\) is endowed with a distance coming from its embedding into the product of Grasmannian manifolds of the corresponding dimensions.

\subsubsection{Dynamical stationary measures}

The partial unstable flag \(E_L(\omega) \in \F_L\) is defined by
\[E_L(\omega)_I = \bigoplus\limits_{i \in I}E_i(\omega),\]
for each \(I \in L\).

We let \(\nu_{L}\) denote the distribution of \(E_L\) i.e. the probability on \(\F_L\) defined by
\[\nu_{L}(A) = m\left(\left\lbrace \omega \in \Omega: E_L(\omega) \in A\right\rbrace \right).\]

\subsubsection{Furstenberg entropy}\label{section:entropy}

Recall that a probability \(\nu\) defined on a space where \(G\) acts continuously is said to by \(\mu\)-stationary if
\[\nu = \int\limits_{G} g_*\nu d\mu(g).\]

A \(\mu\)-stationary probability is ergodic if it is extremal among \(\mu\)-stationary measures.
Given a \(\mu\)-stationary measure \(\nu\) the Furstenberg entropy of the pair \((\mu,\nu)\) is defined as
\[\kappa(\mu,\nu) = \int\limits_{G}\int\limits_{F} \log \frac{dg_*\nu}{d\nu}(gx)d\nu(x) d\mu(g),\]
or \(+\infty\) if \(g_*\nu\) is not absolutely continuous with respect to \(\nu\) for \(\mu\)-a.e. \(g\).

Furstenberg entropy is always non-negative, and is equal to zero if and only if \(g_*\nu = \nu\) for \(\mu\)-a.e. \(g \in G\).  
Since \(\mu\) has finite first moment,  theorem \ref{finiteentropytheorem} implies that   the Furstenberg entropy of the dynamical  stationary measure on a flag space is finite.

\begin{theorem}[Entropy exponent inequality]\label{finiteentropytheorem}
 For each left filtration \(L\), the measure \(\nu_{L}\) is \(\mu\)-stationary, ergodic,  and satisfies
 \[\kappa(\mu,\nu_{L}) \le \sum\limits_{j \notin L(i)}d_i d_j(\chi_i - \chi_j).\] 
\end{theorem}

\subsubsection{Exact dimension}

Let $(X,\rho) $ be a metric space, \( \nu \) a  measure on \(X.\) The {\it {lower dimension }} \( \un \dim \) and the {\it {upper dimension}} \( \ov \dim \) of \((X, \rho, \nu )\) are defined by 
\[ \un \dim \, = \, \underset {\nu}{{\textrm {ess.inf}}} \, \liminf _{r\to 0} \frac{\log \nu (B(x,r))}{\log r} , \quad  \ov \dim \, = \, \underset {\nu}{{\textrm {ess.sup}}} \,\limsup_{r\to 0} \frac{\log \nu (B(x,r))}{\log r}.\]
A measure $\nu $ on $X$ is called {\it {exact dimensional with dimension }}$\delta$ if \(\un \dim = \ov \dim = \delta  .\)

\begin{theorem}[Exact dimension of dynamical stationary measures]\label{exactdimensiontheorem}
Asssume that \( \mu \) has countable support.

Then, for each left filtration \(L\) the measure \(\nu_{L}\) is exact dimensional.  Furthermore there is a formula for its dimension in terms of the Lyapunov exponents of \(\mu\) and a finite set of dynamically defined entropies. 
\end{theorem}

\subsection{Entropy and dimension on configuration spaces}

In order to prove theorem \ref{exactdimensiontheorem} and describe the formula for the dimension of each \(\nu_L\),  we consider a more general family of spaces than spaces of partial flags.  

These spaces are parametrized by certain topologies on the finite set \(\lbrace 1,\ldots,N\rbrace\)
which we call admissible and will now define.

\subsubsection{Admissible topologies}

Given a topology \(T\) on \(\lbrace 1,\ldots,N\rbrace\) we denote by \(T(i)\) the atom of \(i\) for each \(1 \le i \le N\).   A topology is said to be admissible if \(T(i) \subset \lbrace i,i+1,\ldots, N\rbrace\) for all \(i\).

An admissible topology \(T\) is finer than another \(T'\) (denoted by \(T \prec T'\)) if \(T(i) \subset T'(i)\) for all \(i\).   We say \(T\) is one step finer than \(T'\) (denoted by \(T \onefiner T'\)) if \(T \prec T'\), there is a single atom \(T(i)\) with \(T(i) \neq T'(i)\), and furthermore \(T'(i) = T(i) \cup \lbrace j\rbrace\) for some \(j\).

We denote by \(T_1 \prec T_0\) the finest and coarsest admissible topologies.  An admissible topology \(T\) is said to be filtered if it is generated by \(T_0\) and some left filtration \(L\).

Whenever \(T \onefiner T'\) with \(T'(i) = T(i) \cup \lbrace j\rbrace\) we define the Lyapunov exponent \(\chi_{T,T'}\) by
\[\chi_{T,T'} = \chi_i - \chi_j.\]

A monotone path is a sequence \( T^k = T \onefiner T^{k-1} \onefiner \cdots \onefiner T^0 = T'\) such that \(\chi_{T^1,T^0} \le \chi_{T^2,T^2} \le \cdots \le \chi_{T^{k},T^{k-1}}\).

\begin{lemma}[Existence of monotone paths]\label{monotone}
 For each \(T \prec T'\) there at least one monotone path with \(T\) and \(T'\) as its endpoints.
\end{lemma}
\begin{proof}
 This is \cite[Proposition 2.1. part 1]{ll1}.
\end{proof}

\subsubsection{Configuration spaces} 

Given an admissible topology \(T\), we define the (weighted) configuration space \(\X_T\) (with weights \(d_1,\ldots,d_N\)) as the space of sequences \((x_I)_{I \in T}\) such that
\begin{enumerate}
 \item \(x_I\) is a \(d(I)\)-dimensional subspace of \(\R^d\) for each \(I \in T\),
 \item \(x_{I \cup J} = x_I + x_J\) for all \(I,J \in T\), and
 \item \(x_{I\cap J} = x_I \cap x_J\) for all \(I,J \in T\).
\end{enumerate}

Each configuration space \(\X_T\) is endowed with the distance corresponding to its natural embedding in the product of Grassmannian manifolds.  

 If \(T \prec T'\) we denote by \(\pi_{T,T'}: \X_T \to \X_{T'}\) the projection which consists of forgetting the subspaces \(x_I\) for all \(I \in T \setminus T'\). 

 \begin{theorem}[Configuration spaces are fiber bundles]\label{vectorbundletheorem}
  Given admissible topologies \(T \prec T'\) the configuration space \(\X_T\) with the projection \(\pi_{T,T'}\) is locally bilipschitz homeomorphic to a vector bundle over \(\X_{T'}\) with fibers of dimension
  \(\sum\limits_{i = 1}^N \sum\limits_{j \in T'(i) \setminus T(i)} d_i d_j\).
\end{theorem}

The theorem above shows that \(X_{T}\) with the projection \(\pi_{T,T'}\) is a fiber bundle over the base \(X_{T'}\) with fibers which are locally bilipschitz homeomorphic to \(\R^k\) for certain \(k\).  
Let us denote this bundle by \(\X_{T,T'}\) and by \(\X_{T,T'}^{x'} = \pi_{T,T'}^{-1}(x')\) the fiber over \(x' \in \X_{T'}\).

\subsubsection{Dynamical measures on configuration spaces}

Given an admissible topology \(T\), for \(m\)-a.e. \(\omega\) setting
\[E_T(\omega)_I = \bigoplus\limits_{i \in I}E_i(\omega),\]
for all \(I \in T\), defines a configuration in \(\X_T\).   Let \(\nu_T\) be the distribution of \(E_T\) i.e.
\[\nu_T(A) = m\left(\lbrace \omega: E_T(\omega) \in A\rbrace\right).\]

For \(T \prec T'\), clearly \((\pi_{T,T'})_*\nu_T = \nu_{T'}\).  Let \(x' \mapsto \nu_{T,T'}^{x'}\) be a disintegration of \(\nu_{T}\) with respect to the projection \(\pi_{T,T'}\).

\begin{lemma}\label{filteredvsflaglemma}
 Let \(L\) be a left filtration and \(T\) the filtered admissible topology generated by \(T_0\) and \(L\).

Then for \(\nu_{T_0}\)-a.e. \(x_0\), the measure \(\nu_{T,T_0}^{x_0}\) is the bilipschitz image of \(\nu_{L}\) restricted to a full measure set in \(\F_L\).  In particular the measure \(\nu_L\) has the same dimensional properties as the measures \(\nu_{T,T_0}^{x_0}\) for \(\nu_{T_0}\)-a.e. \(x_0\).
\end{lemma}
\begin{proof}[Proof of Lemma \ref{filteredvsflaglemma}]
We say \((x,y) \in \F_L \times \X_{T_0}\) are in general position if \(\dim(x_{I} \cap y_{J}) = d(I \cap J)\) for all \(I \in L, J \in T_0\).   Let \((\F_L \times \X_{T_0})^*\) be the set of pairs in general position and for each \(y \in \X_{T_0}\) let \(\F_L^{y}\) be the set of \(x\) such that \((x,y)\) is in general position.

If \(I \in L\) then \(I = \lbrace 1,\ldots,i\rbrace\) for some \(i\) and 
\[E_L(\omega)_I = (E_1\oplus \cdots \oplus E_i)(\omega).\]

On the other hand if \(J \in T_0\) then \(J = \lbrace k,\ldots,N\rbrace\) for some \(k\) and
\[E_{T_0}(\omega)_I = (E_k\oplus \cdots \oplus E_N)(\omega).\]

It follows that \(E_L(\omega)\) and \(E_{T_0}(\omega)\) are in general position for \(m\)-a.e. \(\omega\).

Since \(E_L(\omega)\) is \(g_{-1}(\omega),g_{-2}(\omega),\ldots\) measurable and \(E_{T_0}(\omega)\) is \(g_0(\omega),g_1(\omega),\ldots\) measurable, they are independent and their joint distribution is \(\nu_L \times \nu_{T_0}\).   Hence, \((\nu_{L} \times \nu_{T_0})((\F_L \times \X_{T_0})^*) = 1\) and \(\nu_L(\F_L^y) = 1\) for \(\nu_{T_0}\)-a.e \(y\).

Notice that each atom  \(K \in T\) is of the form \(I \cap J\) for some atoms \(I \in L\) and \(J \in T_0\).  Let \(F_L: (\F_L \times \X_{T_0})^{*} \to \X_T\) be such that \( F_L(x,y) \) is the unique configuration in \(\X_T \) with  \(F_L(x,y)_{I \cap J} = x_I \cap x_J\) for all  atoms \(I \in L, J \in T_0\).    Observe that \(F_L\) is well defined and bijective since \(x_I = F_L(x,y)_I\) for all atoms \(I \in L\) and \(y_J = F_L(x,y)_J\) for all atoms \(J \in T_0\).

We claim that for each \(y \in \X_{T_0}\) the restriction of \(F_L\) to \(\F_L^{y} \times \lbrace y\rbrace\) is locally bi-lipschitz.   To see this first note that the inverse is a projection forgetting the subspaces \(z_I\) for \(z \in \X_{T}\) if \(I \notin L\) and hence is trivially Lipschitz.   In the forward direction for each atom \(K \in T\) we have \(K = I \cap J\) for some atoms \(I \in L\) and \(J \in T_0\) and therefore \(F_L(x,y)_K = x_I \cap y_J\) which is locally Lipschitz as a function of \(x_I\) in the Grasmannian of dimension \(d(I)\) such that the intersection with the fixed subspace \(y_J\) has dimension \(d(I \cap J)\).

To conclude we observe that since \(\pi_{T,T_0}(F_L(x,y)) = y\) we obtain that \(\nu_{T,T_0}^{y} = (F_L)_*((\nu_L)_{|\F_L^y}\times \lbrace y\rbrace)\) for \(\nu_{T_0}\)-a.e. \(y\).   This concludes the proof of Lemma \ref{filteredvsflaglemma}.
\end{proof}

\subsubsection{Fibered entropy of dynamical  measures}

Let \(T \prec T'\) be  a pair of admissible topologies.
We associate the (fiber) entropy \( \kappa _{T,T'}\)  by the formula
\begin{equation}\label{kappa1} \kappa_{T,T'} = \int\limits_{\Omega} \log \frac{d g_0(\omega)_* \nu_{T,T'}^{E_{T'}(\omega)}}{ d \nu_{T,T'}^{g_0(\omega)E_{T'}(\omega)}}\left(g_0(\omega)E_T(\omega)\right) dm(\omega),\end{equation}
setting \(\kappa_{T,T'} = +\infty\) if the density in the integral fails to exist on a set of positive  \(m\)-measure.
\begin{lemma}\label{entropypropertieslemma}
 The entropies \(\kappa_{T,T'}\) satisfy the following:
 \begin{enumerate}
  \item \(0 \le \kappa_{T,T'} < +\infty\) for all \(T \prec T'\)
  \item \(\kappa_{T,T''} = \kappa_{T,T'} + \kappa_{T',T''}\) for all \(T \prec T' \prec T''\).
  \item For each left filtration \(L\) one has \(\kappa_{T,T_0} = \kappa(\mu,\nu_{L})\) where \(T\) is the filtered admissible topology generated by \(L\) and \(T_0\).
 \end{enumerate}
\end{lemma}

\subsubsection{Exact dimension results}

We are now able to state our two main results, which together imply theorem \ref{exactdimensiontheorem}.

\begin{theorem}[Exact dimension for one step disintegrations]\label{onesteptheorem}
 If \(T \onefiner T'\) are admissible, then for \(\nu_{T}\)-a.e. \(x'\) the measure \(\nu_{T,T'}^{x'}\) is exact dimensional and
 \[\dim(\nu_{T,T'}^{x'}) = \frac{\kappa_{T,T'}}{\chi_{T,T'}}.\]
\end{theorem}

\begin{theorem}[Additivity along monotone paths]\label{additivitytheorem}
  If \(\mu\) has countable support, and 
  \( T^k = T \onefiner T^{k-1} \onefiner \cdots \onefiner T^0 = T'\) is a monotone path of admissible topologies then \(\nu_{T,T'}^{x'}\) is exact dimensional for \(\nu_{T'}\)-a.e. \(x'\) and
  \[\dim(\nu_{T,T'}^{x'}) = \sum\limits_{i = 0}^{k-1}\frac{\kappa_{T^{i+1},T^i}}{\chi_{T^{i+1},T^{i}}}.\]
\end{theorem}

\section{Proof of Theorem \ref{finiteentropytheorem}}\label{theorem:finiteentropytheorem}
\subsection{Proof of stationarity}

For each \(i = 1,\ldots,d\) we have 
\[(E_1\oplus\cdots \oplus E_i)(\omega) = \left\lbrace v \in \R^d: \lim\limits_{n \to +\infty} \frac{1}{n}\log \|g_{-n}^{0}(\omega)v\| \le -\chi_i \right\rbrace.\]

It follows that for each left filtration \(L\) the flag \(E_L\) is \(g_{-1},g_{-2},\ldots\)-measurable.  In particular \(E_L\) and \(g_0\) are independent.  From this we obtain that the distribution of \(g_0(\omega)E_L(\omega)\) is exactly
\[\int\limits_G g_*\nu_L d\mu(g).\]

Let \(\sigma:\Omega \to \Omega\) be the left shift, so that \(g_n(\sigma\omega) = g_{n+1}(\omega)\) for all \(n\).

Since \(m\) is shift invariant and \(g_0(\omega)E_L(\omega) = E_L(\sigma \omega)\) for \(m\)-a.e. \(\omega\), we obtain that the distribution of \(g_0(\omega)E_L(\omega)\) is \(\nu_L\).  This shows that \(\nu_L\) is \(\mu\)-stationary as claimed.

\subsection{Proof of ergodicity}

Set \(\nu_L = \nu_0\) and suppose that \(\nu_0 = t\nu_1 + (1-t)\nu_2\) where \(\nu_i\) is stationary for \(i=1,2\) and \(t \in (0,1)\).

From \cite[Section 1]{furstenberg-glasner} the limits
\[\nu_{i,\omega} = \lim\limits_{n \to +\infty}g_{-n}^0(\omega)_*\nu_i,\]
exists for \(m\)-a.e. \(\omega\) and furthermore
\[\nu_i = \int\limits_{\Omega}\nu_{i,\omega}dm(\omega),\]
for \(i=0,1,2\).

It suffices to show that \(\nu_{0,\omega} = \nu_{1,\omega} = \nu_{2,\omega}\) for \(m\)-a.e. \(\omega\) to conclude that \(\nu_L\) is extremal as claimed.

For this purpose we note that \(E_L(\sigma^{-n}\omega)\) is independent from \(g_{-1}(\omega),\ldots,g_{-N}(\omega)\).  It follows that the conditional distribution of \( g_{-n}^0(\omega)E_L(\sigma^{-n}\omega)\) given \(g_{-1}(\omega),\ldots,g_{-n}(\omega)\) is 
\((g_{-n}^0(\omega))_*\nu_0.\)

Since \(m\)-a.e.  we have \(E_L(\omega) = g_{-n}^0(\omega)E_L(\sigma^{-n}\omega)\) the limit \(\nu_{0,\omega}\) must be the conditional distribution of \(E_L(\omega)\) given \(g_{-1}(\omega),g_{-2}(\omega),\ldots\) which is \(\delta_{E_L(\omega)}\).

Hence we have shown that \(\nu_{0,\omega} = \delta_{E_L(\omega)}\) for \(m\)-a.e. \(\omega\).   Since \(\nu_{0,\omega} = t\nu_{1,\omega} + (1-t)\nu_{1,\omega}\) it follows that \(\nu_{0,\omega} = \nu_{1,\omega} = \nu_{2,\omega}\) for \(m\)-a.e. \(\omega\), which concludes the proof.

\subsection{Proof of entropy estimate}

Let  \(\tF\) denote the space of full flags.  If \(\tmu\) is a probability on \(G\) with finite first moment, and \(\tnu\) is a \(\tmu\)-stationary measure on \(\tF\), we define the Lyapunov exponents of \((\tmu,\tnu)\) so that 
\[\lambda_1(\tmu,\tnu) + \cdots + \lambda_i(\tmu,\tnu) = \int\limits_{G}\int\limits_{\tF}\log |\det_{U_i(x)}(g)| d\tnu(x) d\tmu(g),\]
for \(i = 1,\ldots,d\), where \(|\det_S(g)|\) is the Jacobian of \(g\) restricted to the subspace \(S\), and \(U_i(x)\) denotes the \(i\)-dimensional subspace of \(x\).

Given a left filtration \(L\) we denote by \(\pi_L:\tF \to \F_L\) the natural projection.  For convenience we set \(L(0)\) to be the empty set and \(d(L(0)) = 0\).

\begin{lemma}[Lifts of dynamical measures]\label{stationaryliftlemma}
Let \(L_1\) be the finest left filtration and \(\tnu\) be a \(\mu\)-stationary probability on \(\tF\).   If for all \(i \in \lbrace 1,\ldots,N\rbrace\) one has \(\lambda_k(\mu,\tnu) = \chi_i\) for all \(k \in \lbrace d(L_1(i-1))+1,\ldots,d(L_1(i))\rbrace\) then \((\pi_{L_1})_*\tnu = \nu_{L_1}\).
\end{lemma}
\begin{proof}
 Let \(\tnu_{\omega} = \lim\limits_{n \to +\infty}g_{-n}^0(\omega)_*\tnu\) and suppose, by taking an extension of \((\Omega,m)\), that there exists \(\tE:\Omega \to \tF\) whose conditional distribution given \((g_n(\omega))_{n \in \Z}\) is \(\tnu_{\omega}\) for \(m\)-a.e. \(\omega\).

 Fix \(i \in \lbrace 1,\ldots,N\rbrace\), set \(k = d_1+\cdots+ d_i\) and \(\tE_i(\omega) = U_k(\tE(\omega))\).  We claim that \(\tE_i(\omega) = (E_1 \oplus \cdots \oplus E_i)(\omega)\) for \(m\)-a.e. \(\omega \in \Omega\).  This would prove the lemma since \(\tE\) has distribution \(\tnu\) and we would have \(\pi_{L_1}(\tE(\omega)) = E_{L_1}(\omega)\) for \(m\)-a.e. \(\omega \in \Omega\).
 
 To establish the claim, we first observe that by Birkhoff's theorem
 \[\int\limits_{\Omega}-\lim\limits_{n \to +\infty}\frac{1}{n}\log |\det_{\tE_i(\omega)}g_{-n}^0(\om)^{-1}|dm(\omega) = \int\limits_{\Omega} \log |\det_{\tE_i(\omega)}g_0(\omega)|dm(\omega).\]

 However \(m\)-a.e. one has
 \[\limsup\limits_{n \to +\infty} -\frac{1}{n}\log |\det_{S}(g_{-n}^0(\om)^{-1})| \le d_1\chi_1 + \cdots + d_i \chi_i,\]
 for all \(d_1+\cdots+d_i\)-dimensional subspaces \(S \subset \R^d\).
 
 Therefore
\[\lim\limits_{n \to +\infty}\frac{1}{n}\log |\det_{\tE_i(\omega)}g_{-n}^0(\om)^{-1}| = -(d_1\chi_1+\cdots +d_i\chi_i),\]
so that \(\tE_i(\omega) = E_1(\omega) \oplus \cdots \oplus E_i(\omega)\) for \(m\)-a.e. \(\omega \in \Omega\).   Which concludes the proof.
\end{proof}

To prove Theorem \ref{finiteentropytheorem} we consider a sequence \(\tmu_n\) of probabilities obtained by convolving \(\mu\) with the uniform measure on the ball of radius \(\frac{1}{n}\) around the identity in \(\SO(d)\).

\begin{lemma}[Continuous stationary measures]\label{continuousmulemma}
For each \(n\) there is a unique \(\tmu_n\)-stationary measure \(\tnu_n\) on \(\tF\).   Furthermore \(\tnu\) is absolutely continuous with respect the rotationally invariant probability and for any left filtration \(L\) one has
 \[\kappa(\tmu_n,(\pi_L)_*\tnu_n) = \sum\limits_{j \notin L(i)} \sum\limits_{\substack{d(L(i-1)) < k \le d(L(i))\\ d(L(j-1)) < l \le d(L(j))}}\lambda_k(\tmu_n,\tnu_n) - \lambda_l(\tmu_n,\tnu_n).\]
\end{lemma}
\begin{proof}
 The uniqueness of \(\tnu_n\) and its absolute continuity with respect to the rotationally invariant probability is \cite[Lemma 5.2]{ll1}.   The formula for entropy of an absolutely continuous stationary measure is the calculation following \cite[Proposition 5.8]{ll1}.
\end{proof}

Let \(\tnu_n\) be the unique \(\tmu_n\)-stationary measure on \(\tF\) given by Lemma \ref{continuousmulemma}.  By taking a subsequence we suppose that \(\tnu = \lim\limits_{n \to +\infty}\tnu_n\) exists.

Let \(s_1(g) \ge \cdots \ge s_d(g)\) denote the singular values of a matrix in \(G\).   By construction the integrals 
\[\int\limits_{|\log(s_i(g))| > C}|\log s_i(g)| d\tmu_n(g),\]
go to zero uniformly in \(n\) when \(C \to +\infty\).  This implies (see \cite[Lemma 5.4]{ll1}) that
\[\lambda_k(\mu,\tnu) = \lim\limits_{n \to +\infty}\lambda_k(\tmu_n,\tnu_n).\]

On the other hand by \cite[Theorem 1.1]{froyland} (they deal with a more difficult perturbation where there is worse control on the smallest singular value, their argument goes through for our case) we have
\[\lim\limits_{n \to +\infty}\lambda_k(\tmu_n,\tnu_n) = \chi_i\]
for all \(k \in \lbrace d(L(i-1)+1,\ldots, d(L(i)))\rbrace\).

In particular this implies, by Lemma \ref{stationaryliftlemma} that \((\pi_L)_*\tnu = \nu_L\) for all left filtrations \(L\).

By semi-continuity of Furstenberg entropy we now obtain for all left filtrations \(L\) that 
\[\kappa(\mu,\nu_L) = \kappa(\mu,(\pi_L)_*\tnu) \le \liminf\limits_{n \to +\infty}\kappa(\tmu_{n},(\pi_L)_*\tnu_{n}) = \sum\limits_{j \notin L(i)}d_id_j(\chi_i - \chi_j),\]
which concludes the proof.

\section{Coordinates on configuration bundles}\label{bundlestructure}

In what follows we fix admissible topologies \(T \prec T'\).

We let \(\X_{T,T'}\) denote the space \(\X_{T}\) endowed with the projection \(\pi_{T,T'}\) which we will show in this section is a fiber bundle.

We denote by \(\X_{T,T'}^{x'} = \pi_{T,T'}^{-1}(x')\) the fiber over a point \(x' \in \X_{T'}\).

\subsection{Compatible splittings}

We say a sequence of subspaces \(V = (V_1,\ldots, V_N)\) is a splitting compatible with \(x' \in \X_{T'}\) if
\begin{equation}\label{compatiblesplittingequation}
  x'_{I} = \bigoplus\limits_{i \in I}V_i,
\end{equation}
for all \(I \in T'\).

Notice that in particular this implies \(\dim(V_i) = d_i\) for all \(i\) and \(\R^d = \bigoplus\limits_{i = 1}^N V_i\).

\begin{lemma}[Perpendicular compatible splitting]\label{perpcompatiblesplitting}
 Setting
\begin{equation}
  V_i(x') = x'_{T'(i)} \cap \left(x'_{T'(i)\setminus \lbrace i\rbrace}\right)^{\perp},
\end{equation}
 yields a compatible splitting for each \(x' \in \X_{T'}\).
\end{lemma}
\begin{proof}
  We first observe that since \(x_{T'(i)} = V_i(x') \oplus x_{T'(i)} \setminus \lbrace i\rbrace\) have \(\dim(V_i(x')) = d_i\) for all \(i\).

 We proceed by induction on \(\min(I)\) begining with \(\min(I) = N\) where we trivially have \(V_{N}(x') = x'_{\lbrace N\rbrace}\).  

 Assume that \(\min(I) = i\) and that equation \ref{compatiblesplittingequation} holds for all \(J \in T\) with \(\min(J) > i\). We deduce that
  \begin{align*}
  x'_I &= x'_{T'(i)} + x'_{I \setminus \lbrace i\rbrace} = V_i(x') \oplus x'_{T'(i) \setminus \lbrace i\rbrace} + x'_{I \setminus \lbrace i\rbrace}
   \\ &= V_i(x') \oplus \bigoplus\limits_{j \in T'(i) \setminus \lbrace i\rbrace} V_j(x') + \bigoplus\limits_{j \in I \setminus \lbrace i\rbrace}V_j(x')
   \\ &= V_i(x') + \bigoplus\limits_{j \in I \setminus \lbrace i \rbrace}V_j(x').
\end{align*}
Since the dimensions of the subspaces on the right-hand side add up to \(d(I)\), the sum must be direct as claimed.
\end{proof}

\subsection{Vector spaces of nilpotent mappings}

Given \(T \prec T'\), \(x' \in \X_{T'}\) and \(V\) a compatible splitting we denote by \(\nil_{T,T'}(V)\) the space of linear mapping \(f:\R^d \to \R^d\) such that
\begin{equation}\label{defineVequation}
  f\left(V_i\right) \subset \bigoplus\limits_{j \in T'(i) \setminus T(i)} V_j,
\end{equation}
for \(i = 1,\ldots,N\).

\begin{lemma}\label{dimofnilspaceslemma}
  The dimension of the vector space \(\nil_{T,T'}(V)\) is
  \[\dim(\nil_{T,T'}(V)) = \sum\limits_{i = 1}^N \sum\limits_{j \in T'(i) \setminus T(i)} d_i d_j.\]
\end{lemma}
\begin{proof}
  Let \(S_i = \lbrace f \in \nil_{T,T'}(V):  f_{\mid V_j} = 0\text{ for all }j \neq i\rbrace\).
  This is isomorphic to the space of linear mappings from \(V_i\) to \(\bigoplus\limits_{T'(i) \setminus T(i)}V_j\) so we have
  \[\dim(S_i) = \sum\limits_{j \in T'(i)\setminus T(i)}d_id_j.\]

  Let \(\pi_{V_i}:\R^d \to V_i\) be the projection along \(\bigoplus\limits_{j \neq i}V_j\).   Given \(f \in \nil_{T,T'}(V)\) we have \(f = \sum f\circ \pi_{V_i}\) and \(f\circ \pi_{V_i} \in S_i\).
  Since \(S_i \cap \bigoplus\limits_{j \neq i}S_j = \lbrace 0\rbrace\) it follows that,
  \[\nil_{T,T'}(V) = \bigoplus\limits_{i}S_i,\]
  which concludes the proof.

\end{proof}

\subsection{Fiberwise parametrization}

Given \(T \prec T'\), \(x' \in \X_{T'}\) and a compatible splitting \(V\) we define a mapping
\[\varphi_V: \nil_{T,T'}(V) \to \X_{T,T'}^{x'}\]
by setting
\begin{equation}\label{lineartoconfiguration}
  \varphi_{V}(f)_I = \left(\id + f\right)\left(\bigoplus\limits_{i \in I}V_i(x')\right),
\end{equation}
for all \(I \in T\), where \(\id:\R^d \to \R^d\) is the identity mapping.

\begin{lemma}\label{compatiblephi}
  The sequence \(\varphi_V(f)\) defined above is a configuration in \(\X_{T,T'}^{x'}\) for each \(f \in \nil_{T,T'}(V)\).
\end{lemma}
\begin{proof}
  First notice that \(\varphi_{V}(0)_I = \bigoplus\limits_{i \in I}V_i\) for all \(I \in T\). So in this case, it follows from equation \ref{compatiblesplittingequation} that \(\varphi_{V}(0) \in \X_{T}^{x'}\).

  For general \(f \in \nil_{T,T'}(V)\) we observe that, since \(f\) is nilpotent, \(\id+f\) is a linear isomorphism of \(\R^d\).

  Since \(\varphi_{V}(f)_I = (\id + f)(\varphi_{V}(0)_I)\) for all \(I \in T\) this suffices to show that \(\varphi_V(f)\) is a configuration in \(\X_T\).

  We now verify that \(\varphi_V(f)_{T'(i)} = x'_{T'(i)}\) for all \(i\).  Since the case \(f = 0\) has been established, it suffices to show that \(\id + f\) leaves each \(x'_{T'(i)}\) invariant.

  To do this, using the case \(f = 0\), we calculate
  \begin{align*}
    (\id+f)(x'_{T'(i)}) &= (\id+f)\left(\bigoplus\limits_{T'(i)}V_j\right) 
               \\ &\subset \bigoplus\limits_{T'(i)}V_j + f\left(\bigoplus\limits_{T'(i)} V_j\right)
               \\ &\subset \bigoplus\limits_{T'(i)}V_j + \sum\limits_{T'(i)}\bigoplus\limits_{k \in T'(j) \setminus T(j)}V_k
               \\ &= \bigoplus\limits_{T'(i)}V_j
               \\ &= x'_{T'(i)}
  \end{align*}

  Since the dimensions of the subspaces on the left and right-hand side are equal this shows the desired claim.
\end{proof}

\begin{lemma}\label{injectivephi}
  The mapping \(\varphi_V: \nil_{T,T'}(V) \to \X_{T,T'}^{x'}\) defined above is injective.
\end{lemma}
\begin{proof}
Suppose that \(f \neq g\), let \(i\) be the largest index such that \(f\) and \(g\) do not coincide on \(V_i\), and \(v \in V_i\) be such that \(f(v) \neq g(v)\).

If \(\varphi_{V}(f) = \varphi_{V}(g)\) then in particular
\[\left(\id + f\right)\left(V_i\oplus S\right) = \left(\id +g\right)\left(V_i\oplus S\right),\]
  where \(S = \bigoplus\limits_{j \in T(i)\setminus\lbrace i\rbrace}V_j\).

Therefore there must exist \(v' \in V_i\) and \(w \in S\) such that \(v + f(v) = v' + w + g(v') + g(w)\).

By equation \ref{defineVequation}, \(f(V_i)\) and \(g(V_i)\) are complementary to \(V_i \oplus S\).  Therefore we must have \(v = v'\) and \(f(v) - g(v) = w+g(w)\).

However, for the same reason this implies \(w = 0\) and hence \(f(v) = g(v)\).  So \(\varphi_{T,T'}\) is injective as claimed.
\end{proof}

\begin{lemma}\label{surjectivephi}
  The mapping \(\varphi_V: \nil_{T,T'}(V) \to \X_{T,T'}^{x'}\) defined above is surjective.
\end{lemma}
\begin{proof}
  Let \(x \in \X_{T}^{x'}\), we say \(f \in \nil_{T,T'}(V)\) is \(i\)-good if
\[\left(\id + f_i\right)\left(\bigoplus\limits_{j \in J}V_j(x')\right) = x_{J},\]
for all \(J \in T\) with \(\min(J) > i\). 

Notice that \(f = 0\) is \(N-1\)-good.

Suppose that \(f\) is \(i\)-good we will show that there exists \(g \in \nil_{T,T'}(V)\) with \(g_{\mid V_j} = 0\) for all \(j \neq i\) such that \(f+g\) is \((i-1)\)-good.

To see this consider the subspace
\[S = \bigoplus\limits_{j \in T'(i) \setminus T(i)}V_j.\]

We claim that
\[S \oplus x_{T(i) \setminus \lbrace i\rbrace} = x_{T'(i) \setminus \lbrace i\rbrace}.\]

To see this first notice that by equation \ref{compatiblesplittingequation} we have
\[S \oplus \bigoplus\limits_{j\in T(i) \setminus \lbrace i\rbrace} V_j= x'_{T'(i) \setminus \lbrace i\rbrace}.\]

Now, since \(\id + f\) is a linear isomorphism leaving \(x'_{T'(i) \setminus \lbrace i\rbrace}\) invariant it suffices to show that \((\id+f)(V_j) \cap S = \lbrace 0\rbrace\) for each \(j \in T(i) \setminus \lbrace i\rbrace\).

Fix \(j \in T(i) \setminus \lbrace i\rbrace\) and \(v \in V_j\).  By hypothesis \(f(v) \in \bigoplus\limits_{k \in T'(j) \setminus T(j)}V_k\) so unless \(v = 0\) one has that \(v + f(v)\) has a non-zero component in \(V_j\).  Since \(v + f(v) \in x'_{T'(i) \setminus \lbrace i\rbrace} = \bigoplus\limits_{k \in T'(i)\setminus \lbrace i\rbrace}V_k\) we conclude that if \(v+f(v) \in S\) then \(v = 0\).  Which proves the claim.

From the claim it follows that
\[V_i \oplus S \oplus x_{T(i) \setminus \lbrace i\rbrace} = x_{T'(i)}.\]

Since \(x_{T(i)} \cap x_{T'(i) \setminus \lbrace i\rbrace} = x_{T(i) \setminus \lbrace i\rbrace}\) this implies that \(x_{T(i)} \cap (V_i(x') \oplus S)\) is \(d_i\)-dimensional.

Therefore, for each \(v \in V_i\) there exists a unique \(w \in S\) such that \(v + w \in x_{T(i)}\).   Setting \(g(v) = w\) we have that \(f+g\) coincides with \(f\) when restricted to \(\bigoplus\limits_{j > i}V_j(x')\).

By construction \(x_{T(i)}\) is the image under \(\id+f+g\) of \(V_i \oplus S\), so \(f+g\) is \((i-1)\)-good as required.
\end{proof}

\subsection{Proof of theorem \ref{vectorbundletheorem}}

We define for each \(x' \in \X_{T'}\) the perpendicular compatible splitting \(V(x') = (V_1(x'),\ldots, V_N(x'))\) as in lemma \ref{perpcompatiblesplitting}.

We define
\begin{equation}\label{vectorbundleequation}
  \V_{T,T'} = \lbrace (x',f):  x' \in \X_{T'},   f \in \nil_{T,T'}(V(x'))\rbrace.
\end{equation}

This is a vector bundle with base \(\X_{T'}\) with the projection onto the first coordinate.  It is a sub-bundle of the product \(\X_{T'} \times \text{Hom}(\R^d,\R^d)\).  We metrizice it with the sum of the distance in \(\X_{T'}\) and any operator norm on \(\text{Hom}(\R^d,\R^d)\).

We claim that \(\varphi_{T,T'}:\V_{T,T'} \to \X_{T,T'}\) be defined by
\[\varphi_{T,T'}(x',f) = \varphi_{V(x')}(f),\]
is a locally bilipschitz homeomorphism.

For each \(i\) the mapping \(x' \mapsto V_i(x')\) is lipschitz since it is defined by taking the perpendicular complement of a subspace of \(x'\) within another.  Since, \(\varphi_{V}\) is defined applying \(\id +f\) to sums of the subspaces \(V_i(x')\) it is a lipschitz mapping as well.  Hence \(\varphi_{T,T'}\) is a lipschitz mapping.

From lemma \ref{compatiblephi}, \ref{injectivephi}, and \ref{surjectivephi} it follows that \(\varphi_{T,T'}\) has an inverse which we denote by \(\psi_{T,T'}\).  From the proof of lemma \ref{surjectivephi}, it follows that \(\psi_{T,T'}\) is locally lipschitz because the mapping associating in a splitting \(S \oplus S'\) to each subspace \(S''\) complementary to \(S'\) a linear mapping \(f:S \to S'\) whose graph is \(S''\) is locally lipschitz from the \(\dim(S)\)-dimensional Grasmannian to \(\text{Hom}(S\oplus S', S\oplus S')\).

\subsection{An example of a non-trivial configuration bundle}

In view of theorem \ref{vectorbundletheorem} it is natural to ask whether the vector bundles constructed are trivial (i.e. the product of the base times \(\R^k\)).  We give an example showing this is not the case.

\begin{example}
  For \(N = 2, d_1=1,d_2=2\) for the only two admissible topologies \(T_1 \prec T_0\) the bundle \(\X_{T_1,T_0}\) is homeomorphic to the tangent bundle of the projective space \(\mathcal{P}(\R^3)\).
\end{example}

We note that \(\X_{T_1}\) is the space of ordered pairs \((L,P)\) of a one-dimensional subspace \(L\) and a two dimensional subspace \(P\) in \(\R^3\) in general position. 
The map \(\pi_{T_1,T_0}\) is the projection onto the second coordinate.

To establish the claim we endow \(\mathcal{P}(\R^3)\) with the unique rotationally invariant Riemannian metric for which the angle in \(\R^3\) between the subspaces \(x \in \mathcal{P}(\R^3)\) and \(\exp_x(v)\) is \(\|v\|\) for all \((x,v)\) in the tangent bundle \(T\mathcal{P}(\R^3)\), where \(\exp\) denotes the Riemannian exponential mapping.

Notice that \((L,P) \in \X_{T_1}\) if and only if the angle between \(L\) and the perpendicular complement \(L'\) of \(P\) is less than \(\frac{\pi}{2}\).

Hence, for each \((L,P) \in \X_{T_1}\) there is a unique \(v \in T_{L'}\mathcal{P}(\R^3)\) with \(\|v\| < \frac{\pi}{2}\) such that \(L = \exp_{L'}(v)\).

This identifies \(\X_{T_1,T_0}\) with the bundle of tangent vectors with norm less than \(\frac{\pi}{2}\) which is diffeomorphic to the tangent bundle itself.

\section{Fiberwise changes of coordinates}\label{algebraic}

We consider in what follows admissible topologies \(T \prec T'\).  We say two splittings \(V=(V_1,\ldots,V_N)\) and \(W = (W_1,\ldots,W_N)\) are \(T'\)-compatible if there exists \(x' \in \X_{T'}\) such that both \(V\) and \(W\) are 
compatible with \(x'\).  Equivalently
\[\bigoplus\limits_{I}V_i = \bigoplus\limits_{I}W_i,\]
for all \(I \in T'\).

Given two such splittings we let \(\pi_{V_i}\) and \(\pi_{W_i}\) be the corresponding systems of projections, so \(\pi_{V_i}\) is the projection onto \(V_i\) along \(\bigoplus\limits_{j \neq i}V_i\) and similarly for \(\pi_{W_i}\).

\begin{theorem}\label{changeofcoordinatestheorem}
  For each pair of admissible topologies \(T \prec T'\) there exists a polynomial with integer coefficients \(p_{T,T'}\) in \(2N+1\) non-commuting variables such that
  \[\varphi_W^{-1}\circ \varphi_V(f) = p_{T,T'}(f,\pi_{V_1},\ldots, \pi_{V_N},\pi_{W_1},\ldots,\pi_{W_N}),\]
  for all \(T'\)-compatible splittings \(V,W\), and all \(f \in \nil_{T,T'}(V)\).
\end{theorem}

\subsection{Projections}

\begin{lemma}\label{projectionruleslemma}
  For each \(i\) one has 
  \[\pi_{V_i} = \pi_{V_i}\pi_{W_i}\pi_{V_i}, \pi_{W_i} = \pi_{W_i}\pi_{V_i}\pi_{W_i}\]
  and furthermore,
  \[\pi_{W_i}\pi_{V_j} = \pi_{V_i}\pi_{W_j} = 0,\]
  for all \(j > i\).
\end{lemma}
\begin{proof}
  We show that \(\pi_{V_i} = \pi_{V_i}\pi_{W_i}\pi_{V_i}\), and \(\pi_{V_i}\pi_{W_j} = 0\).  The other claims follow by permuting \(V\) and \(W\).

  Since \(V\) and \(W\) are compatible, there exist \(S,S'\) such that 
  \[S = \bigoplus\limits_{j \ge i}V_j = \bigoplus\limits_{j \ge i}W_j,\]
  and
  \[S' = \bigoplus\limits_{j > i}V_j = \bigoplus\limits_{j > i}W_j.\]

  It follows that both \(\pi_{V_i}\) and \(\pi_{W_i}\) are projections with kernel \(S'\) when restricted to \(S\).

  Hence, \(\pi_{V_i}\pi_{W_i} = \pi_{V_i}\) within \(S\), and therefore \(\pi_{V_i}\pi_{W_i}\pi_{V_i} = \pi_{V_i}^2 = \pi_{V_i}\), as claimed.

  If \(j > i\), then \(\pi_{W_j}(v) \subset S'\) for all \(v\), so one has \(\pi_{V_i}\pi_{W_j} = 0\) as claimed.
\end{proof}

\subsection{Proof of theorem \ref{changeofcoordinatestheorem}}

We begin by defining a family polynomials in the non-commutative variables \(x,\alpha_1,\ldots,\alpha_N,\beta_1,\ldots,\beta_N\).

We do this inductively by first setting 
\[p_{1,i} = 0 = q_{1,i}\]
for \(i = 1,\ldots,N\).

We continue for each \(i\) defining for \(j = 2,3,\ldots,N\) according to the rules
    \begin{enumerate}
      \item If \(j \in T(i)\) set \(p_{j,i} = 0\) and \(q_{j,i} = \sum\limits_{k \le j}\alpha_j p_{k,i}\),
      \item If \(j \notin T(i)\) set \(q_{j,i} = -\alpha_j r(x) \beta_i + \sum\limits_{l \in T'(i) \setminus T(i)} \sum\limits_{l \le k < j} \alpha_j r(x)x\alpha_k p_{l,i}\),
   \end{enumerate}
   where \(r(x) = 1-x+x^2+\cdots + (-x)^{N-1}\).

Since \(g = \sum\limits_{i,j}\pi_{W_j}g\pi_{W_i}\) the theorem follows  from the following claim by setting \(p_{T,T'} = \sum\limits_{i,j} p_{j,i}\):
\begin{claim}\label{mainlemma}
Let \(f \in \nil_{T,T'}(V)\), \(g = \varphi_{W}^{-1}\varphi_V(f) \in \nil_{T,T'}(W)\), and \(X = (f,\pi_{V_1},\ldots,\pi_{V_N},\pi_{W_1},\ldots, \pi_{W_N})\).

Then for all \(i,j\) one has
   \[p_{i,j}(X) = \pi_{W_i}g\pi_{W_j}\text{ and }q_{i,j}(X) = \pi_{V_i}g\pi_{W_j}.\]
\end{claim}

As a first step to proving the lemma we show:
\begin{lemma}\label{basecaselemma}
  One has \(\pi_{V_1}g\pi_{W_i} = \pi_{W_1}g\pi_{W_i} = 0 = p_{1,i}(X) = q_{1,i}(X)\) for all \(i\).
\end{lemma}
\begin{proof}
  This follows immediately since \(g\pi_{W_i}\) takes values in
  \[\bigoplus\limits_{T'(i) \setminus T(i)}W_j \subset \bigoplus\limits_{j \ge 2}W_j = \bigoplus\limits_{j \ge 2}V_j.\]
\end{proof}

We now cover the case where \(j \in T(i)\).
\begin{lemma}\label{easycaselemma}
  If \(j \in T(i)\) then \(\pi_{W_j}g\pi_{W_i} = 0 = p_{j,i}(X)\) and
  \[\pi_{V_j}g\pi_{W_i} = \sum\limits_{k \le j}\pi_{V_j}\pi_{W_j}g\pi_{W_i}.\]
\end{lemma}
\begin{proof}
  Since \(g\pi_{W_i}\) takes values in \(\bigoplus\limits_{T'(i) \setminus T(i)}W_k\) the first claim is immediate.

  For the second we use lemma \ref{projectionruleslemma} to conclude
  \[\pi_{V_j}g\pi_{W_i} = \sum\limits_{k}\pi_{V_j}\pi_{W_k}g\pi_{W_i} = \sum\limits_{k \le j}\pi_{V_j}\pi_{W_k}g\pi_{W_i}.\]
\end{proof}

Finally we cover the case \(j \notin T(i)\).
\begin{lemma}\label{hardcaselemma}
  If \(j \notin T(i)\) then
  \[\pi_{V_j}g\pi_{W_i} = -\pi_{V_j}r(f)\pi_{W_i} + \sum\limits_{l \in T'(i) \setminus T(i)}\sum\limits_{l \le k < j}\pi_{V_j}r(f)f\pi_{V_k}\pi_{W_l}g\pi_{W_i},\]
  and
  \[\pi_{W_j}g\pi_{W_i} = \sum\limits_{k \le j}\pi_{W_j}\pi_{V_k}g\pi_{W_i}.\]
\end{lemma}
\begin{proof}
  The second claim follows by the same argument as in lemma \ref{easycaselemma}.

  For the first we observe that, since \(\varphi_W(g) = \varphi_V(f)\) we have
  \[(\id+g)\bigoplus\limits_{T(i)}W_k = (\id+f)\bigoplus\limits_{T(i)}V_k.\]

  Since \(j \notin T(i)\) this implies
  \[\pi_{V_j}(\id+f)^{-1}(\id+g)\pi_{W_i} = 0.\]

  Since \(f \in \nil_{T,T'}(V)\) we have \(f^N = 0\) so \((\id+f)^{-1} = r(f) = \id - r(f)f\).  It follows that
  \[\pi_{V_j}g\pi_{W_i} = -\pi_{V_j}r(f)\pi_{W_i} + \pi_{V_j}r(f)fg\pi_{W_i}.\]

  We now write \(g = \sum\limits_{l,k}\pi_{V_k}\pi_{W_l}g\) and use lemma \ref{projectionruleslemma} to obtain
  \[\pi_{V_j}g\pi_{W_i} = -\pi_{V_j}r(f)\pi_{W_i} + \sum\limits_{l \in T'(i) \setminus T(i)}\sum\limits_{l \le k \le j}\pi_{V_j}r(f)f\pi_{V_k}\pi_{W_l}g\pi_{W_i}.\]

  However, since \(r(f)f(V_j) \subset V_{j+1}\oplus \cdots \oplus V_N\) we obtain \(\pi_{V_j}r(f)f\pi_{V_k} = 0\) by lemma \ref{projectionruleslemma}, so in fact
  \[\pi_{V_j}g\pi_{W_i} = -\pi_{V_j}r(f)\pi_{W_i} + \sum\limits_{l \in T'(i) \setminus T(i)}\sum\limits_{l \le k < j}\pi_{V_j}r(f)f\pi_{V_k}\pi_{W_l}g\pi_{W_i},\]
  as claimed.
\end{proof}

Claim \ref{mainlemma} now follows by induction:  the base case for \(p_{1,i}\) and \(q_{1,i}\) is given by lemma \ref{basecaselemma}.
Assuming that claim \ref{mainlemma}  has been established for \(p_{k,i}\) and \(q_{k,i}\) with \(k < j\) then if \(j \in T(i)\), it holds for \(p_{j,i}\) and \(q_{j,i}\) by lemma \ref{easycaselemma},
and otherwise it holds by lemma \ref{hardcaselemma}.

\subsection{Change of coordinates for one step bundles}

\begin{lemma}[One step change of coordinates]\label{onestepchangelemma}
  If \(T \onefiner T'\) then for all \(x' \in \X_{T'}\) and all compatible splittings \(V,W\) the mapping
  \[\varphi_W^{-1}\circ \varphi_{V}:\nil_{T,T'}(V) \to \nil_{T,T'}(W),\]
  is affine.
  Furthermore, setting \(\varphi_W^{-1}(\varphi_V(f)) = A(f) + B\) with \(A\) linear and \(B \in \nil_{T,T'}(W)\) one has
  \[A(f) =  \pi_{W_j} f \pi_{W_i},\]
  where \(i < j\) are such that \(T(i) = T'(i) \setminus \lbrace j\rbrace\).
\end{lemma}
\begin{proof}
  Let \(i < j\) be the elements such that \(T'(i) \setminus T(i) = \lbrace j\rbrace\), and let
  \[S = \bigoplus\limits_{k \ge j}V_i = \bigoplus\limits_{k \ge j}W_i.\]

  From the definition of \(\nil_{T,T'}(V)\), we have \(f(\R^d) \subset S\) and since \(f^2 = 0\) we have \(f(S) = \lbrace 0\rbrace\).

  From lemma \ref{projectionruleslemma} every projection \(\pi_{V_i}(S) \subset S\) and \(\pi_{W_i}(S) \subset S\) for all \(i\).
  Therefore, we have \(faf = 0\) for all \(a\) obtained by composition of projections along the splittings \(V\) and \(W\).

  Let \(p_{T,T'}\) be the non-commuting polynomial given by theorem \ref{changeofcoordinatestheorem}.

  From the above observation it follows that any term of 
  \[p_{T,T'}(x,\alpha_1,\ldots,\alpha_N,\beta_1,\ldots,\beta_N)\] 
  containing \(x\) twice or more, evaluates to \(0\) when calculating \(p_{T,T'}(f,\pi_{V_1},\ldots,\pi_{W_N})\).   Hence, \(\varphi_W^{-1}\varphi_V\) is affine as claimed.

  Since \(f \in \nil_{T,T'}(V)\) we have \(f = \pi_{V_j} f \pi_{V_i}\).  Setting \(g = \varphi_{W}^{-1}(\varphi_V(f))\) we have \(g = \pi_{W_j} g \pi_{W_i}\) and therefore 
  \[A(f) + B = \pi_{W_j}p_{T,T'}(\pi_{V_j}f\pi_{V_i},\pi_{V_1},\ldots,\pi_{W_N})\pi_{W_i}.\]
 
  This implies that \(A(f)\) is equal to a finite sum of terms of the form
  \[\pi_{W_j}\pi_{X_{1}}\cdots \pi_{X_{m}}\pi_{V_j}f\pi_{V_i}\pi_{Y_1}\cdots \pi_{Y_n}\pi_{W_i},\]
  for some sequence of subspaces \(X_k,Y_k\) chosen from the splittings \(V\) and \(W\).

  Since \(\pi_{V_k}^2 = \pi_{V_k}\) and \(\pi_{W_k}^2 = \pi_{W_k}\) for all \(k\), we may assume \(X_k \neq X_{k+1}\) for \(k = 1,\ldots, m-1\).

  By, lemma \ref{projectionruleslemma}, if the term is non-zero there is a non-increasing sequence \(j \ge l_1 \ge l_2 \ge \ldots \ge l_m\) such that \(X_k = V_{l_k}\) if \(k\) is odd and \(X_k = W_{l_k}\) if \(k\) is even.

  For the last term \(\pi_{X_m}\pi_{V_j}\) to be non-zero we must have \(l_m \le j\), so we obtain that \(l_k = j\) for all \(k\) and since, within \(\bigoplus\limits_{l \ge j}V_l\) the projections \(\pi_{W_j}\) and \(\pi_{V_j}\)
  have the same kernel it follows that
  \[\pi_{W_j}\pi_{X_1}\cdots \pi_{X_m}\pi_{V_j} = \pi_{W_j}\pi_{V_j}.\]

  Similarly, \(\pi_{V_i}\pi_{Y_1}\cdots \pi_{Y_n}\pi_{W_i} = \pi_{V_i}\pi_{W_i}\) and we obtain that
  \[A(f) = C \pi_{W_j}\pi_{V_j}f\pi_{V_i}\pi_{W_i} = C \pi_{W_j}f\pi_{W_i},\]
  for some integer \(C\) which does not depend on the particular compatible splitting \(W\).

  Since, when \(V = W\) we have \(C = 1\) we have shown \(A(f) = \pi_{W_j}f\pi_{W_i}\) as claimed.
\end{proof}

\subsection{Change of coordinates for filtered topologies}

\begin{lemma}[Filtered change of coordinates]\label{filteredchangelemma}
  If \(T\) is filtered then for all \(x' \in \X_{T'}\) and all compatible splittings \(V,W\) the mapping
  \[\varphi_W^{-1}\circ \varphi_{V}:\nil_{T,T'}(V) \to \nil_{T,T'}(W),\]
  is affine.
\end{lemma}
\begin{proof}
The key observation is that \(T'(i) \setminus T(i) \in T'\).

To see this, let \(T\) be generated by \(T_0\) and the left filtration \(L\) and notice that \(T(i) = T_i(i) \cap L(i) = \lbrace i, i+1,\ldots, i+k\rbrace\) for certain \(k\).  This implies
\[T'(i) \setminus T(i) = T'(i) \cap \lbrace k+1,\ldots, N\rbrace \in T',\]
as claimed.

As a consequence we obtain
\[x'_{T'(i) \setminus T(i)} = \bigoplus\limits_{T'(i) \setminus T(i)}V_k = \bigoplus\limits_{T'(i) \setminus T(i)}W_j.\]

Now let \(g_0 = \varphi_W^{-1}(\varphi_V(0))\) so that
\[\varphi_W(g_0)_{T(i)} = (\id + g_0)\left(\bigoplus\limits_{T(i)}W_j\right) = \bigoplus\limits_{T(i)}V_j = \varphi_V(0)_{T(i)},\]
for all \(i\).  We claim that
\[\varphi_W^{-1}(\varphi_V(f)) = g_0 + f(\id + g_0),\]
for all \(f \in \nil_{T,T'}(V)\).

We first verify that \(g_0 + f(\id + g_0) \in \nil_{T,T'}(W)\) by checking
\[(g_0 + f(\id +g_0))W_i \subset g_0(W_i) + f(\bigoplus\limits_{T(i)}V_j) = x'_{T'(i) \setminus T(i)} + \sum\limits_{j \in T(i)}\bigoplus\limits_{T'(j) \setminus T(j)}V_k.\]

As noted before \(T(i) = \lbrace i,\ldots, k\rbrace\) for a certain \(k\) and \(L(i) = \lbrace 1,\ldots, k\rbrace\).  
It follows that for each \(j \in T(i)\) we have \(T(j) = \lbrace j,\ldots, k\rbrace\).
Hence \(T'(j) \setminus T(j) = T'(j) \cap \lbrace k,\ldots, N\rbrace \subset T'(i) \cap \lbrace k,\ldots, N\rbrace = T'(i) \setminus T(i)\).
This shows that 
\[\sum\limits_{j \in T(i)}\bigoplus\limits_{T'(j) \setminus T(j)}V_k \subset x'_{T'(i) \setminus T(i)},\]
so that \(g_0 + f(\id + g_0) \in \nil_{T,T'}(W)\) as claimed.

To conclude we verify
\begin{align*}
  \varphi_W(g_0 + f(\id+g_0))_{T(i)} &= (\id + g_0 + f(\id + g_0))\bigoplus\limits_{T(i)}W_j
                      \\ &= (\id + f)(\id + g_0)\bigoplus\limits_{T(i)}W_j
                      \\ &= (\id + f)\bigoplus\limits_{T(i)}V_j 
                      \\ &= \varphi_V(f)_{T(i)},
\end{align*}
for all \(i\).  So that \(\varphi_W(g_0 + f(\id + g_0)) = \varphi_V(f)\) as required.
\end{proof}

\subsection{Examples of change of coordinates}

\subsubsection{Non-affine change of coordinates}

\begin{example}
For \(N = 4, d_1=d_2=d_3=d_4 = 1, d = 4\) and \(T \prec T'\) defined by
\[T'(1) = \lbrace 1,2,3,4\rbrace, T'(2) = \lbrace 2,3,4\rbrace, T'(3) = \lbrace 3,4\rbrace, T'(4) = \lbrace 4\rbrace,\] 
and
\[T(1) = \lbrace 1,3\rbrace, T(2) = \lbrace 2,3\rbrace, T(3) = \lbrace 3\rbrace, T(4) = \lbrace 4\rbrace,\] 
the change of coordinates \(\varphi_W^{-1}\circ \varphi_V\) bewtween compatible splittings is not affine.
\end{example}

Fix a basis \(v_1,v_2,v_3,v_3\) of \(\R^4\) and let \(V = (V_1,V_2,V_3,V_4)\) be the splitting where \(V_i\) is generated by \(v_i\).

If we let \(x'_{I} = \bigoplus\limits_{I}V_i\) for each \(I \in T'\) then \(V\) is compatible with \(x'\).

We let \(w_1 = v_1+v_2+v_3+v_4, w_2 = v_2+v_3+v_4, w_3 = v_3+v_4, w_4 = v_4\) and \(W\) be the splitting where \(W_i\) is generated by \(w_i\).
By construction \(W\) is compatible with \(x'\) as well.

For \(f \in \nil_{T,T'}(V)\) let \(A\) be the associated matrix in the basis \(\lbrace v_i\rbrace\).   It has the form
\[A = \begin{pmatrix}
  0 & 0 & 0 & 0
\\ a & 0 & 0 & 0
\\ 0 & 0 & 0 & 0
\\ b & c & d & 0
\end{pmatrix}.\]

We notice that since \(A^3 = 0\) one has
\[(\id + A)^{-1} = \id - A + A^2 =  \begin{pmatrix}
  1 & 0 & 0 & 0
\\ -a & 1 & 0 & 0
\\ 0 & 0 & 1 & 0
\\ ac - b & -c & -d & 1
\end{pmatrix}.\]

We let \(B\) be the matrix associated to the identity from basis \(\lbrace w_i\rbrace\) to \(\lbrace v_i\rbrace\) which is
\[B = \begin{pmatrix}
   1  &  0 &  0 & 0
\\ 1 &  1 &  0 & 0
\\  1 & 1 &  1 & 0
\\  1 &  1 & 1 & 1
\end{pmatrix}.\]

Finally, we let \(X\) be the matrix associated to \(g = \varphi_{W}^{-1}(\varphi_V(f))\) in the basis \(\lbrace w_i\rbrace\), which will have the form
\[X = \begin{pmatrix}
  0 & 0 & 0 & 0
\\ w & 0 & 0 & 0
\\ 0 & 0 & 0 & 0
\\ x & y & z & 0
\end{pmatrix}.\]

By definition we must have
\begin{equation}\label{bruteforceequation}\left\lbrace\begin{array}{l}
    (\id + f)^{-1}(\id+g)W_4 = V_4
\\   (\id + f)^{-1}(\id+g)W_3 =V_3
\\   (\id + f)^{-1}(\id+g)(W_2\oplus W_3) = (V_2 \oplus V_3)
\\   (\id + f)^{-1}(\id+g)(W_1\oplus W_3) = (V_1 \oplus V_3)
\end{array}\right..\end{equation}

To solve for \(X\) we calculate
\[(\id + A)^{-1}B(\id + X) = \begin{pmatrix}
  1 & 0 & 0 & 0
\\ -a+w+1 & 1 & 0 & 0
\\ w+1 & 1 & 1 & 0
\\ ac - b + w(-c-d+1) -c-d+x+1 &  -c-d+y+1 & -d+z+1 & 1
\end{pmatrix}.\]

The equation \ref{bruteforceequation} becomes
\[\left\lbrace\begin{array}{l}
    -d+z+1 = 0
\\  -c-d+y+1 = 0
\\ -a+w+1 = 0
\\ ac - b + w(-c-d+1) - c-d+x+1 = 0.
\end{array}\right..\]

We have thus obtained:
\[X = \begin{pmatrix}
  0 & 0 & 0 & 0
\\ a-1 & 0 & 0 & 0
\\ 0 & 0 & 0 & 0
\\ -a+b+ad & c+d-1 & d-1 & 0
\end{pmatrix}.\]

Here it can be seen explicitly that \(\varphi_W^{-1}\circ \varphi_V\) is not an affine mapping from \(\nil_{T,T'}(V)\) to \(\nil_{T,T'}(W)\), and instead is polynomial with degree \(2\).

\subsection{Dimension \(3\)}

\begin{example}
  For \(N = 3 = d\) and \(d_1=d_2=d_3=1\) the change of coordinates between compatible splittings \(V,W\) is affine for all pair of admissible topologies \(T \prec T'\).
\end{example}

The only case not covered by lemma \ref{onestepchangelemma} or lemma \ref{filteredchangelemma} is
\[T = \lbrace 1,3\rbrace,\lbrace 2\rbrace,\lbrace 3 \rbrace \prec T' = \lbrace 1,2,3\rbrace, \lbrace 2,3\rbrace, \lbrace 3\rbrace,\]
where we have specified the topologies by listing their atoms.

We will now show that for this particular choice of \(T \prec T'\) changes of coordinates between compatible splittings are also affine.

As before to see this we pick two basis \(v_1,v_2,v_3\) and \(w_1,w_2,w_3\) of \(\R^3\) such that if \(V = (V_1,V_2,V_3)\) and \(W = (W_1,W_2,W_3)\) are the associated splittings there is \(x' \in \X_{T'}\) such that
\[x'_I = \bigoplus\limits_{I}V_i = \bigoplus\limits_{I}W_i,\]
for all \(I \in T'\).

Given \(f \in \nil_{T,T'}(V)\) we let \(A\) be its matrix from the basis \(\lbrace v_i\rbrace\) to itself.  We have
\[A = \begin{pmatrix}
0 & 0 & 0\\
a & 0 & 0\\
0 & b & 0
\end{pmatrix},
\]
for certain \(a,b\) and 
\[(\id + A)^{-1} = \id - A + A^2 = 
  \begin{pmatrix}
1 & 0 & 0\\
-a & 1 & 0\\
ab & -b & 1
\end{pmatrix}.\]

We may change the basis \(\lbrace w_i\rbrace\) while still generating the same splitting \(W\) in such a way that \(w_1 = v_1 + \alpha v_2 + \beta v_3, w_2 = v_2 + \gamma v_3\) and \(w_3 = v_3\).  
After doing this, the matrix \(B\), associated to the identity from the basis \(\lbrace w_i\rbrace\) to \(\lbrace v_i\rbrace\) will have the form
\[B = 
  \begin{pmatrix}
1 & 0 & 0\\
\alpha & 1 & 0\\
\beta & \gamma & 1
\end{pmatrix}.\]

We now let \(X\) be the matrix associated to \(g = \varphi_W^{-1}(\varphi_V(f))\) in the basis \(\lbrace w_i\rbrace\).  We have
\[X = 
  \begin{pmatrix}
0 & 0 & 0\\
x & 0 & 0\\
0 & y & 0
\end{pmatrix}.\]

The element \(g \in \nil_{T,T'}(W)\) is determined by the system of equations
\begin{equation}\label{brute2}
  \left\lbrace\begin{array}{l} 
      (\id + f)^{-1}(\id + g)W_3 = V_3 \\
      (\id + f)^{-1}(\id + g)W_2 = V_2 \\
      (\id + f)^{-1}(\id + g)W_1 \oplus W_3 = V_1 \oplus V_3
  \end{array}\right.
\end{equation}

We calculate
\[(\id + A)^{-1}B(\id + X) = 
  \begin{pmatrix}
1 & 0 & 0\\
-a + \alpha + x & 1 & 0\\
ab - \alpha b + x (\gamma -b) + \beta & - b + \gamma + y  & 1
\end{pmatrix}.\]

Equation \ref{brute2} becomes
\[
\left\lbrace\begin{array}{l} 
      -b + \gamma +y = 0\\
      -a + \alpha + x = 0\\
  \end{array}\right.
\]

So we obtain
\[X = 
  \begin{pmatrix}
0 & 0 & 0\\
a-\alpha & 0 & 0\\
0 & b - \gamma & 0
\end{pmatrix}.\]

\section{Dynamics on configuration spaces}\label{lineardynamics}

\subsection{Inner products on one-step fibers\label{innerproductsection}}

Let \(T \onefiner T'\) and \(i < j\) be such that \(T(i) = T'(i) \setminus \lbrace j\rbrace\).

Consider, for each \(x' \in \X_{T'}\) the perpendicular compatible splitting \(V(x')\) and the vector bundle \(\V_{T,T'}\) defined by equation \ref{vectorbundleequation}.

We fix on each fiber 
\[\V_{T,T'}^{x'} = \lbrace (x',f): f \in \nil_{T,T'}(V(x'))\rbrace,\]
the Frobenius norm \(\| \|_F\) (square root of sums of squares of singular values) of the restriction of \(f\) to \(x'_{T'(i)}\), and the associated inner product and distance.

Given a linear mapping \(A: \V_{T,T'}^{x_1'} \to \V_{T,T'}^{x_2'}\), we denote by \(s_1(A) \ge \cdots \ge s_{d_id_j}(A)\) its singular values with respect to these inner products.

\subsection{Dynamics on one-step bundles}

When \(T \prec T'\) are given we set \(x'(\omega) = E_{T'}(\omega), x(\omega) = E_{T}(\omega)\) and let \(V(\omega)\) be the perpendicular \(x'(\omega)\)-compatible splitting given by lemma \ref{perpcompatiblesplitting}.

\begin{lemma}[Affine action for one step bundles]\label{affineonesteplemma}
  If \(T \onefiner T'\) then 
  \[F_m^n(\omega) = \varphi_{V(\sigma^{n}(\omega))}^{-1}g_{m}^n(\omega)\varphi_{V(\sigma^m(\omega))}: \nil_{T,T'}(V(\sigma^m(\omega))) \to \nil_{T,T'}(V(\sigma^n(\omega)))\]
  is affine and setting
  \[F_m^n(\omega)(f) = A_{m}^n(\omega) f + B_m^n(\omega),\]
  one has
  \[\int\limits_{\Omega}\left|\log\left(s_k(A_0^1(\omega))\right)\right| dm(\omega) < +\infty\]
  and
  \[\chi_{T,T'} = -\lim\limits_{n \to +\infty}\frac{1}{n}\int\limits_{\Omega} \log\left(s_k(A_0^n(\omega))\right) dm(\omega),\]
  for \(k = 1,\ldots, d_id_j\) where \(i < j\) are such that \(T(i) = T'(i) \setminus \lbrace j\rbrace\).
\end{lemma}
\begin{proof}
  Fixing \(\omega \in \Omega\) set \(g = g_0(\omega), x' = x'(\omega)\), and \(W = gV(x')\).

  From lemma \ref{linearactionlemma} below we have
  \[\varphi_{gV(x')}^{-1}\varphi_{V(x')}(f) = gfg^{-1},\]
  for all \(f \in \V_{T,T'}^{x'}\).

  By lemma \ref{onestepchangelemma}, we have that \(\varphi_{V(gx')}^{-1}\varphi_{W}\) is affine.  Using the explicit form of the linear part we obtain that \(\varphi_{V(gx')}^{-1}\varphi_{V(x')}(f) = A(f) + B\) where
  \[A(f) = \pi_{V_j(gx')}gfg^{-1}\pi_{V_i(gx')}.\]

  We notice that \(\|\pi_{V_i(gx')}\|_F = 1\) on \(gx'_{T'(i)}\), and \(\|\pi_{V_j(g')}\|_F = 1\) on \(gx'_{T'(j)}\).  Since \(f(x'_{T'(i)}) \subset x'_{T'(j)}\) we obtain
  \[\|\varphi_{V(gx')}^{-1}\varphi_{V(x')}(f)\|_F \le \|g\|_F\|g^{-1}\|_F\|f\|_F.\]

  This implies that \(\|A\| \le \|g\|_F\|g^{-1}\|_F\).  Applying the same argument to \(g^{-1}\) we obtain \(s_{d_id_j}(A) \ge \|g\|_F^{-1}\|g^{-1}\|_F^{-1}\).

  It follows from \(m\)-integrability of the logarithm of singular values of \(g_0(\omega)\) that \(\log s_k(A_0^1(\omega))\) is \(m\)-integrable for all \(k\) as claimed.

  We will now show that
  \[\chi_{T,T'} = -\lim\limits_{n \to +\infty}\frac{1}{n}\log\left(s_k(A_0^n(\omega))\right),\]
  for \(m\)-a.e. \(\omega \in \Omega\).  By the sub-additive ergodic theorem this implies the remaining claim on the integrals of the right hand side above.

  For this purpose we let \(\theta(\omega)\) be the minimal angle between \(E_k(\omega)\) and \(\bigoplus\limits_{l \neq k}E_l(\omega)\) over all \(k\).
  From equation \ref{Oseledets} of Oseledets theorem we have
  \[\lim\limits_{n \to +\infty} \log \sin(\theta(\sigma^n(\omega))) = 0,\]
  for  \(m\)-a.e. \(\omega \in \Omega\).

  Let \(W(\omega)\) be the Oseledets splitting so \(W_i(\omega) = E_i(\omega)\) for \(i = 1,\ldots,N\) and \(V(\omega)\) be the perpendicular splitting adapted to \(x'(\omega) = E_{T'}(\omega)\).

  Since \(g_0^n(\omega)W(\omega) = W(\sigma^n(\omega))\) we may use lemma \ref{linearactionlemma} applied to \(W\), and lemma \ref{onestepchangelemma} applied to the pairs \(V(\omega),W(\omega)\) and \(V(\sigma^{n}(\omega)), W(\sigma^{n}(\omega))\) to obtain
  \[A_0^n(\omega)(f) = \pi_{V_j(\sigma^{n}(\omega))}g_0^n(\omega)\pi_{E_j(\omega)}f\pi_{E_i(\omega)}g_0^{n}(\omega)^{-1}\pi_{V_i(\sigma^n(\omega))},\]
  for all \(f \in \V_{T,T'}^{x'(\omega)}\).

  We observe that 
  \begin{enumerate}
   \item \(\pi_{V_j(\sigma^n(\omega))}\) has norm \(1\) on \(E_j(\sigma^n(\omega)) = g_0^n(\omega)E_j(\omega) \subset x'(\sigma^n(\omega))_{T'(j)}\) since it is a perpendicular projection on the latter subspace,
   \item  By Oseledets theorem the norm of \(g_0^n(\omega)\) on \(E_j(\omega)\) is bounded above by \(\exp(n\chi_j + a_n)\) for some sequence with \(\lim\limits_{n \to +\infty}a_n/n = 0\),
   \item The norm of \(\pi_{E_j}(\omega)\) on \(f(\R^d) = V_j(\omega)\) is at most \(\sin(\theta(\omega))^{-1}\).
   \item Since \(g_0^n(\omega)E_i(\omega) = E_i(\sigma^n(\omega))\) we have \(\pi_{E_i(\omega)}g_0^n(\omega)^{-1} = \pi_{E_i(\omega)}g_0^n(\omega)^{-1}\pi_{E_i(\sigma^n(\omega))}\).
   \item By Oseledets theorem the norm of the latter transformation is bounded above by \(\exp(-n\chi_i + b_n)\) for some sublinear sequence \(b_n\). 
   \item The norm of \(\pi_{E_i(\sigma^n(\omega))}\) on \(V_i(\sigma^n(\omega))\) is at most \(\sin(\theta(\sigma^n(\omega)))^{-1}\).
   \item The norm of \(\pi_{V_i(\sigma^n(\omega))}\) is \(1\) on \(g_0^n(\omega)x'(\omega)_{T'(i)}\).
  \end{enumerate}

  Combining these facts we obtain that
  \[\|A_0^n(\omega)(f)\| \le \|f\|\exp(-n(\chi_i - \chi_j)) \exp(a_n + b_n)\sin(\theta(\omega))^{-1}\sin(\theta(\sigma^n(\omega)))^{-1},\]
  so that
  \[\limsup\limits_{n \to +\infty}\frac{1}{n}\log\|A_0^n(\omega)\| \le -\chi_{T,T'}.\]

  For the lower bound we observe that there is a unit vector \(v \in V_i(\sigma^n(\omega))\) such that
  \[\|f \pi_{E_i(\omega)}g_0^n(\omega)^{-1}v\| = \|f\|\|\pi_{V_i(\omega)}\pi_{E_i}(\omega)g_0^n(\omega)^{-1}v\|.\]

  The minimal norm of \(\pi_{V_i(\omega)}\pi_{E_i(\omega)}g_0^n(\omega)^{-1}\) on \(V_i(\sigma^n(\omega))\) is bounded from below (by Oseledets theorem) by \(\sin(\theta(\omega))\exp(n\chi_i + a_n')\) for some sublinear sequence \(a_n'\).
 
  Similarly, the minimal norm of \(\pi_{V_j(\sigma^{n}(\omega))}g_0^n(\omega)\pi_{E_j(\omega)}\) on \(V_j(\omega)\) is bounded below by \(\sin(\theta(\sigma^n(\omega)))\exp(-n\chi_j + b_n')\) for somesublinear sequence \(b_n'\). 

  Hence,
  \[\|A_0^n(\omega)fv\| \ge \|f\|\exp(-n(\chi_i - \chi_j))\exp(a_n'+b_n')\sin(\theta(\sigma^n(\omega)))\sin(\theta(\omega)),\]
  which yields
  \[\liminf\limits_{n \to +\infty}\frac{1}{n}\log s_{d_id_j}(A_0^n(\omega)) \ge -\chi_{T,T'},\]
  for \(m\)-a.e. \(\omega \in \Omega\).  Which concludes the proof.
\end{proof}

\subsection{Coordinates for linear action}

We note that if \(V = (V_1,\ldots,V_N)\) is adapted to \(x'\) then \(g^{-1}V = (g^{-1}V_1,\ldots, g^{-1}V_N)\) is an adapted splitting for \(g^{-1}x'\).
For the coordinates given by these two splittings the action of \(g\) linear between the corresponding fibers.

\begin{lemma}[Linearizing coordinates]\label{linearactionlemma}
  For each \(x' \in \X_{T'}\) and each adapted splitting \(V\) one has
  \[\varphi_{V}^{-1}\circ g \circ \varphi_{g^{-1}V}(f) = gfg^{-1},\]
  for all \(g \in G\) and all \(f \in \nil_{T,T'}(g^{-1}V)\).

  In particular \(\varphi_{V}^{-1}\circ g\circ \varphi_{g^{-1}V}:\nil_{T,T'}(g^{-1}V) \to \nil_{T,T'}(V)\) is linear.
\end{lemma}
\begin{proof}
  We check that \(\varphi_V(gfg^{-1}) = g\varphi_{g^{-1}V}(f)\) as follows
  \begin{align*}
    \varphi_V(gfg^{-1})_I &= \left(\id + gfg^{-1}\right)\left(\bigoplus\limits_{i \in I}V_i\right)
                       \\ &=  g\left(\id + f\right)g^{-1}\left(\bigoplus\limits_{i \in I}V_i\right)
                       \\ &=  g\left(\id + f\right)\left(\bigoplus\limits_{i \in I}g^{-1}V_i\right)
                       \\ &=  g\varphi_{g^{-1}V}(f).
  \end{align*}
\end{proof}

\begin{corollary}
  If either \(T \onefiner T'\) or if \(T\) is filtered then the action of \(G\) on \(\X_{T,T'}\) is affine on each fiber.
\end{corollary}
\begin{proof}
  This is immediate from lemma \ref{linearactionlemma} combined with either lemma \ref{filteredchangelemma} or lemma \ref{onestepchangelemma}.
\end{proof}

\subsection{Dynamics on configuration bundles}

Given \(T \prec T'\) we will now prove that it is possible to approximate \(E_{T}(\omega)\) given \(E_{T'}(\omega),g_{-1}(\omega),\ldots,g_{-n}(\omega)\) to within a distance decreasing with rate
\begin{equation}\label{bestratequation}
  \underline{\chi}_{T,T'} = \min\limits_i \min\limits_{j \in T'(i) \setminus T(i)}\chi_i - \chi_j.
\end{equation}

\begin{lemma}[Approximation of Oseledets configurations]\label{approximationlemma}
  Let \(T \prec T'\) be admissible and \(\underline{\chi} = \underline{\chi}_{T,T'}\) be defined by equation \ref{bestratequation}, then 
  \[\limsup\limits_{n \to +\infty}\|\varphi_{V(\omega)}^{-1}g_{-n}^0(\omega)\varphi_{V(\sigma^{-n}(\omega))} - \varphi_{V(\omega)}^{-1}E_T(\omega)\| \le -\underline{\chi},\]
  for \(m\)-a.e. \(\omega \in \Omega\).
  
  In particular
  \[E_T(\omega) = \lim\limits_{n \to +\infty}g_{-n}^0(\omega)\varphi_{V(\sigma^{-n}(\omega))}(0),\]
  for \(m\)-a.e. \(\omega \in \Omega\).
\end{lemma}
\begin{proof}
  We let \(W(\omega) = (E_1(\omega),\ldots, E_N(\omega))\) be the Oseledets splitting.  Applying lemma \ref{linearactionlemma} we obtain that
  \[\varphi_{W(\omega)}^{-1}g_{-n}^0(\omega)\varphi_{W(\sigma^{-n}(\omega))}(f) = g_{-n}^0(\omega)fg_{-n}^0(\omega)^{-1}.\]

  Suppose \(i <j\) are such that \(j \in T'(i) \setminus T(i)\), and let \(S_{i,j}(\sigma^{-n}(\omega))\) be the subspace of functions \(f \in \nil_{T,T'}(W(\sigma^{-n}(\omega)))\) such that \(f(E_i(\sigma^{-n}(\omega))) \subset E_j(\sigma^{-n}(\omega))\) and \(f(E_k(\sigma^{-n}(\omega))) = \lbrace 0\rbrace\) for all \(k \neq i\).

  As in lemma \ref{dimofnilspaceslemma} we have that \(\nil_{T,T'}(W(\sigma^{-n}(\omega))) = \bigoplus S_{i,j}(\sigma^{-n}(\omega))\), so one has given \(f \in \nil_{T,T'}(W(\sigma^{-n}(\omega)))\) a decomposition
  \[f = \sum\limits_{i}\sum\limits_{j \in T'(i) \setminus T(i)}f_{i,j},\]
  where \(f_{i,j} = \pi_{E_j(\sigma^{-n}(\omega))}f\pi_{E_i(\sigma^{-n}(\omega))}\).

  Let \(\theta(\omega)\) be the minimal angle between \(E_k(\omega)\) and \(\bigoplus\limits_{l \neq k}E_l(\omega)\) over all \(k\).  We have
  \[\|f_{i,j}\| \le \sin(\theta(\sigma^{-n}(\omega)))^{-2}\|f\|.\]

  By Oseledets theorem there exists a sublinear sequence \(a_n\) such that the norm and minimal norm of \(g_{-n}^0(\omega)\) on \(E_k(\sigma^{-n}(\omega))\) are within \(\exp(\pm a_n)\) of \(\exp(n\chi_k)\) for all \(k\) and \(n\).

  Hence we obtain
  \begin{align*}
    \|\varphi_{W(\omega)}^{-1}g_{-n}^0(\omega)\varphi_{W(\sigma^{-n}(\omega))}(f)\| &\le \sum\limits_{i}\sum\limits_{j \in T'(i)\setminus T(i)}\exp(2a_n)\exp(-n(\chi_i - \chi_j))\|f_{i,j}\|
                                                                                 \\ &\le N^2\exp(2a_n)\sin(\theta(\sigma^{-n}(\omega)))^{-2}\|f\|.
\end{align*}

This implies (since the angle \(\theta(\sigma^{-n}(\omega))\) is not exponentially small by Oseledets theorem) that
\[\limsup\limits_{n \to +\infty}\frac{1}{n}\log\|\varphi_{W(\omega)}^{-1}g_{-n}^0(\omega)\varphi_{W(\sigma^{-n}(\omega))}\| \le - \underline{\chi}.\]

By theorem \ref{changeofcoordinatestheorem} we have that \(\varphi_{V(\omega)}^{-1}\varphi_{W(\omega)}(f)\) is a finite sum of terms obtained as composition of \(f\) and projections along \(W(\omega)\) and \(V(\omega)\). We let \(P(\omega)(f)\) be the sum of terms where at least one \(f\) appears.

We now calculate
\begin{align}\label{someaproxeqn}
  \|\varphi_{V(\omega)}^{-1}g_{-n}^0(\omega)&\varphi_{V(\sigma^{-n}(\omega))}(0) - \varphi_{V(\omega)}^{-1}E_T(\omega)\| =
                                         \\ &= \|P(\omega)\left(\varphi_{W(\omega)}^{-1}g_{-n}^{0}(\omega)\varphi_{W(\sigma^{-n}(\omega))}B_n(\omega)\right)\|,
\end{align}
where
\[B_n(\omega) = \varphi_{W(\sigma^{-n}(\omega))}\varphi_{V(\sigma^{-n}(\omega))}(0).\]

We claim that \(\|B_n(\omega)\| \le \exp(b_n)\) for some sublinear sequence \(b_n\) depending on \(\omega\).  

To establish the claim we observe that by theorem \ref{changeofcoordinatestheorem} one has that \(B_n(\omega)\) is a finite sum of terms obtained by composition of the projections along \(V(\sigma^{-n}(\omega))\) and the Oseledets splitting \(W(\sigma^{-n}(\omega))\).  By Oseledets theorem the norms of the projections along \(W(\sigma^{-n}(\omega))\) (which are controlled by \(\sin(\theta(\sigma^{-n}(\omega)))^{-1}\)) are bounded by \(\exp(c_n)\) for some subexponential sequence \(c_n\).

For the perpendicular splitting the same holds.  To see this, first observe that \(\pi_{V_i(\sigma^{-n}(\omega))}\) is an orthogonal projection within \(x'(\sigma^{-n}(\omega))_{T'(i)}\) and therefore
\[\|\pi_{V_i(\sigma^{-n}(\omega))}\| \le 1,\]
on this subspace.

Next, restricted to \(S_i = \bigoplus\limits_{j \ge i}V_j(\sigma^{-n}(\omega)) = \bigoplus\limits_{j \ge i}E_j(\sigma^{-n}(\omega))\) both \(\pi_{V_i}(\sigma^{-n}(\omega))\) and \(\pi_{V_j}(\sigma^{-n}(\omega))\) have the same kernel and therefore
\[\|\pi_{V_i}(\sigma^{-n}(\omega))\| = \|\pi_{V_i}(\sigma^{-n}(\omega))\pi_{E_i}(\sigma^{-n}(\omega))\| \le \|\pi_{E_i}(\sigma^{-n}(\omega))\|,\]
on this subspace.

We establish the claim bounding the norm of \(\pi_{V_i}\) on all of \(\R^d\) as follows:
\begin{align*}
\|\pi_{V_i(\sigma^{-n}(\omega))}\| &= \|\pi_{V_i(\sigma^{-n}(\omega))}(\id - \pi_{V_{i-1}(\sigma^{-n}(\omega))})\cdots (\id - \pi_{V_1(\sigma^{-n}(\omega))})\|
\\ &\le \|\pi_{E_i(\sigma^{-n}(\omega))}\|(1 + \|\pi_{E_{i-1}(\sigma^{-n}(\omega))}\|)\cdots (1 + \|\pi_{E_1(\sigma^{-n}(\omega))}\|).
\end{align*}

Using the claim we obtain that \(P(\omega)\) is evaluated in the right-hand side of equation \ref{someaproxeqn} above, at a vector with norm at most
\(\exp(-n\underline{\chi})\) up to sub-exponential terms.  Since \(P(\omega)(0) = 0\) and \(P(\omega)\) is lipschitz in a neighborhood of \(0\), this concludes the proof of the lemma.
\end{proof}

\section{Fibered entropy of dynamical  measures}\label{entropy}

Let \(T \prec T'\) be  a pair of admissible topologies.
We associate the (fiber) entropy \( \kappa _{T,T'}\)  by the formula
\begin{equation} \kappa_{T,T'} = \int\limits_{\Omega} \log \frac{d g_0(\omega)_* \nu_{T,T'}^{E_{T'}(\omega)}}{ d \nu_{T,T'}^{g_0(\omega)E_{T'}(\omega)}}\left(g_0(\omega)E_T(\omega)\right) dm(\omega).\end{equation}
\begin{proposition}
 The entropies \(\kappa_{T,T'}\) satisfy the following:
 \begin{enumerate}
  \item \(\kappa_{T,T'} \ge 0\) for all \(T \prec T'\)
  \item \(\kappa_{T,T''} = \kappa_{T,T'} + \kappa_{T',T''}\) for all \(T \prec T' \prec T''\).
 % \item For each left filtration \(L\) one has \(\kappa_{T,T_0} = \kappa(\mu,\nu_{L})\) where \(T\) is the filtered admissible topology generated by \(L\) and \(T_0\).
 \end{enumerate}
\end{proposition}

All the claimed properties follow from the interpretation of \(\kappa(\mu,\nu_L)\) and \(\kappa_{T,T'}\) as conditional mutual information (see lemma \ref{furstenbergentropylemma} below).  We follow \cite[Section 5.2]{ll1} and \cite[Section 2]{lessa}.

Recall that given random variables \(X,Y\) with individual distributions \(\mu_X,\mu_Y\) and joint distribution \(\mu_{X,Y}\) the mutual information
\[I(X,Y) = \int \log \frac{d\mu_{X,Y}}{d\mu_X\times \mu_Y}(x,y) d\mu_{X,Y}(x,y),\]
if the Radon-Nikodym derivative if the integral exists and \(+\infty\) otherwise (see \cite{pinsker}, \cite{wyner}).

If \(X,Y,Z\) are measurable functions from \(\Omega\) into Polish spaces, the conditional mutual information \(I(X,Y|Z)\) between \(X\) and \(Y\) given \(Z\) is the expected value of the mutual information obtained using the definition above for a family of regular conditional distributions \(\mu_{X|Z},\mu_{Y|Z}, \mu_{X,Y|Z}\).  Here we have deviated from the notation of \cite{ll1} where \(I(X,Y|Z)\) denoted the \(Z\)-measurable random variable whose expected value integral gives the conditional mutual information. Recall that if \(W,X,Y,Z\) are measurable functions from \(\Omega\) into Polish spaces, then one has the following chain rule
\begin{equation*}  I(W,(X,Y)\vert Z)\;=\; I(W,X\vert Y,Z) +  I(W,Y\vert Z).\end{equation*}
With these notations, proposition \ref{entropypropertieslemma} reduces to the following lemma

\begin{lemma}[Furstenberg entropy and mutual information]\label{furstenbergentropylemma}
 For each left filtration \(L\) one has \(\kappa(\mu,\nu_L) = I(g_{-1},E_L)\).  For each pair of admissible topologies \(T \prec T'\) one has \(\kappa_{T,T'} = I(g_{-1},E_T|E_{T'})\).
 
 In particular \(0 \le \kappa_{T,T'}\) for all admissible \(T \prec T'\) and \(\kappa_{T,T''} = \kappa_{T,T'}+\kappa_{T',T''}\) for all admissible \(T \prec T' \prec T''\). 
\end{lemma}
\begin{proof}
We first show that \(\kappa(\mu,\nu_L) = I(g_{-1},E_L)\) for each left filtration \(L\).  Keeping with the notation above we let \(\mu_X\) denote the distribution of a random element defined on \((\Omega,m)\).

For this purpose notice first that, by shift invariance of \(m\), the joint distribution of \((g_{-1},E_L)\) coincides with that of \((g_0,g_0E_L)\).

Since \(g_0\) and \(E_L\) are independent we have \(\mu_{g_0,E_L} = \mu \times \nu_L\) and it follows that
\[\int\limits_{G}\int\limits_{\F_L} f(g,x) d\mu_{g_0,g_0E_L}(g,x) = \int\limits_G \int\limits_{\F_L} f(g,x) dg_*\nu_L(x) d\mu(g).\]

Hence, either \(I(g_{-1},E_L) = +\infty\) or one has
\[I(g_{-1},E_L) = \int\limits_{G}\int\limits_{F_L} \log \frac{dg_*\nu_L}{d\nu_L}(x) dg_*\nu_L(x) d\mu(g).\]
In both cases \(I(g_{-1},E_L) = \kappa(\mu,\nu_L)\) as claimed.

\

For a general admissible topology, we have, by \cite[Proposition 6.2]{ll1} (up to slight changes of notations),
\begin{equation}\label{entropyT} \kappa_{T,T_0} = I(g_{-1},E_T|E_{T_0}).\end{equation} 

The chain rule for mutual information implies that if \(T \prec T'\) one has 
\begin{align*}
I(g_{-1},E_T|E_{T_0}) &= I(g_{-1},(E_T,\pi_{T,T'}(E_T))|E_{T_0})
\\ &= I(g_{-1},\pi_{T,T'}(E_T)|E_{T_0}) + I(g_{-1},E_T|\pi_{T,T'}(E_T),E_{T_0})
\\ &= I(g_{-1},E_{T'}|E_{T_0}) + I(g_{-1},E_T|E_{T'}),
\end{align*}
where in the last line we use that \(E_{T_0} = \pi_{T',T_0}(E_{T'})\) so conditioning on \(E_{T'}\) and \(E_{T_0}\) is the same as conditioning only on \(E_{T'}\).

Directly from the definition using the densities of the disintegration of \(\nu_T\) with respect to \(\pi_{T,T'}\) and \(\pi_{T,T_0}\) one sees that \(\kappa_{T,T_0} = \kappa_{T',T_0} + \kappa_{T,T'}\).  This shows that \(\kappa_{T,T'} = I(g_{-1},E_{T}|E_{T'})\) for all \(T \prec T'\).

From the chain rule for mutual information \(\kappa_{T,T''} = \kappa_{T,T'} + \kappa_{T',T''}\) for all admissible \(T \prec T' \prec T''\).
\end{proof}

%It remains to prove Lemma \ref{entropypropertieslemma}, part 3, by identifying  \(\kappa(\mu,\nu_L)\) with \(\kappa_{T,T_0}\) when \(T\) is filtered.

Using theorem \ref{finiteentropytheorem} and  lemma \ref{furstenbergentropylemma}, we obtain.
\begin{corollary}[Finiteness of conditional entropies]\label{kappafinite}
 For all admissible \(T \prec T'\) one has \(\kappa_{T,T'} < +\infty\).
\end{corollary}

Another consequence  of lemma \ref{furstenbergentropylemma} is 
\begin{corollary}
 If \(L\) is a left filtration and \(T\) is generated by \(T_0\) and \(L\) then \(\kappa(\mu,\nu_L) = \kappa_{T,T_0}\).
\end{corollary}
\begin{proof}
 From lemma \ref{furstenbergentropylemma} we have
 \[\kappa_{T,T_0} = I(g_{-1},E_T|E_{T_0}).\]
 
 From lemma \ref{filteredvsflaglemma} one has \(E_T = F(E_L,E_{T_0})\) for some measurable invertible function \(F\).  If follows that
 \[I(g_{-1},E_T|E_{T_0}) = I(g_{-1},E_L|E_{T_0}).\]
 
 Since \(E_{T_0}\) is independent from \(g_{-1}\) and \(E_L\) we obtain
 \[I(g_{-1},E_L|E_{T_0}) = I(g_{-1},E_L) = \kappa(\mu,\nu_L),\]
 as claimed.
\end{proof}

\

Let  \( \Om, \PP \)  be a probability space and \( (X,Y,Z)\) three random variables. If the relative mutual information  \(I(X,Y|Z ) = \E_{x,y,z}[\log \frac{d\PP_{(X,Y)|Z}^z}{d\PP _{X|Z}^z d\PP_{Y|Z}^z}(x,y)]\) is finite, we can write the density \( \frac{d\PP_{(X,Y)|Z}^z}{d\PP _{X|Z}^z d\PP_{Y|Z}^z}(x,y)\) in terms of the density of the conditional measures given \((Y,Z)\):
\(  \frac{d\PP_{(X,Y)|Z}^z}{d\PP _{X|Z}^z d\PP_{Y|Z}^z}(x,y)=  \frac{d\PP_{X|(Y,Z)}^{y,z}}{d\PP_{X|Z}^z} (x) .\)

By corollary \ref{kappafinite},  \( \kappa _{T, T_0} = I(g_{-1}, E_T|E_{T_0})\) is finite.  The projection \( \om \in \Om  \mapsto E_T(\om ) \in \X_T \) admits disintegrations that we denote \( m_T^{x}.\) We still denote by \( m_T^{x}\) the projection  of \( m_T^{x}\)  to \(\Om_-\).
%It follows that the density \( f \) of the measure \( m_T^{E_T(\om )}\) restricted to \( g_{-1}(\om ) \) with  respect to \( \mu \) is given by \[f(\om ) :=\,\frac{dm_T^{E_T(\om)}|_{g_{-1}}}{d\mu} (g_{-1}(\om ))\,=\, \frac{ d(g_{-1} (\om ))_\ast\nu _T^{(g_{-1}(\om ))^{-1} E_+(\om )} }{  d\nu _T^{E_+(\om ) } }(E_T(\om )) .\]
Since \(E_{T_0}(\om)  \) is independent from \( g_{-1}(\om) \) and is measurable with respect to \(E_T(\om),\) we obtain the  following  formulas for \( \kappa _{T, T_0} :\)
\begin{equation}\label{kappa3} \kappa _{T, T_0} \; = \;   \int \left(\log\frac{dm_T^{E_T(\om)}|_{g_{-1}}}{d\mu} (g_{-1}(\om ))\right) \, dm(\om )\
\end{equation} 
\begin{proposition}\label{kappa4}  We have, for \(m\)-a.e. \(\om \),  \[\kappa _{T, T_0} = \lim\limits _{n \to \infty } \frac{1}{n} \log \frac{dm_T^{E_T(\om )}|_{\{g_{-1}, \ldots, g_{-n}\}}}{d\otimes _1^n \mu} (g_{-1}(\om ), \ldots, g_{-n}(\om )).\] \end{proposition}
\begin{proof} Apply Birkhoff ergodic theorem to the integrand in (\ref{kappa3}). See \cite[Proposition 6.4]{ll1}. \end{proof}
\begin{corollary}\label{kappa6} Let \( T \prec T' \) be admissible topologies.   We have, for \(m\)-a.e. \(\om \),  \[\kappa _{T, T'} = \lim\limits _{n \to \infty } \frac{1}{n} \log \frac{dm_T^{E_T(\om )}|_{\{g_{-1}, \ldots, g_{-n}\}}}{dm_{T'}^{E_{T'}(\om )}|_{\{g_{-1}, \ldots, g_{-n}\}}}.\] \end{corollary}

In the case when the measure \( \mu\) is discrete, the  formula in corollary  \ref{kappa6} becomes, for \(m\)-a.e. \(\om \),    \begin{equation}\label{kappa7} \kappa _{T, T'} = \lim\limits _{n \to \infty } \frac{1}{n} \log \frac{m_T^{E_T(\om )}([g_{-1}(\om), \ldots, g_{-n}(\om)])}{m_{T'}^{E_{T'}(\om )}([g_{-1}(\om), \ldots, g_{-n}(\om)])},\end{equation}
where, for \( (g_{-1}, \ldots, g_{-n}) \in G^n, \)  \([g_{-1}, \ldots, g_{-n}]\) is the cylinder \( \{ \{ h_q \}_{q \in \Z} \in \Om: h_q = g_q {\textrm{ for }} -n \leq q \leq -1 \}.\)

\section{Exact dimension on one step bundles}\label{onestep}

We fix from now on \(T \onefiner T'\) and \(i < j\) so that \(T(i) = T'(i) \setminus \lbrace j\rbrace\).  To simplify notation set
\[x(\omega) = E_T(\omega), x'(\omega) = E_{T'}(\omega), \nu_\omega = \nu_{T,T'}^{x'(\omega)}, \kappa = \kappa_{T,T'}\text{ and }\chi = \chi_{T,T'}.\]

We use \(V(\omega)\) for the perpendicular splitting compatible with
\(x'(\omega)\) given by lemma \ref{perpcompatiblesplitting} and endow 
\(\X_{T,T'}^{x'(\omega})\) with the distance induced by \(\varphi_{V(\omega)}\) and 
the norm defined in section \ref{innerproductsection}.
Let \(B_\omega(x,r)\) denote the ball of radius \(r\) centered at \(x \in \X_{T,T'}^{x'(\omega)}\) with respect to this norm.

We recall that \(A_m^n(\omega)\) is the linear part of the affine mapping \(\varphi_{V(\sigma^{n}(\omega))}g_m^n(\omega)\varphi_{V(\sigma^m(\omega))}\) as in lemma \ref{affineonesteplemma}. 
\begin{lemma}[Main lemma]\label{mainonesteplemma}
  If
  \[\int\limits_{\Omega}\log(s_{d_id_j}(A_0^{1}(\omega))) dm(\omega) = -I_{d_id_j} \le -I_1 = \int\limits_{\Omega}\log(s_1(A_0^{1}(\omega)))dm(\omega) < 0,\]
  then
  \[\frac{\kappa}{I_{d_id_j}} \le \underline{\dim}(\nu) \le \overline{\dim}(\nu) \le \frac{\kappa}{I_1}.\]
\end{lemma}

\subsection{Proof of theorem \ref{onesteptheorem}}

We prove theorem \ref{onesteptheorem} assuming lemma \ref{mainonesteplemma}.

Given \(\epsilon > 0\) by lemma \ref{affineonesteplemma} we may pick \(K\) so that
\[-(\chi + \epsilon)K \le \int\limits_{\Omega}s_{d_id_j}(A_0^K(\omega))dm(\omega) \le \int\limits_{\Omega}s_1(A_0^K(\omega)) dm(\omega) \le -(\chi-\epsilon)K.\]

If \(\epsilon < \chi\) then we may apply lemma \ref{mainonesteplemma} changing our measure \(m\) to \((\mu^{*K})^\Z\) where \(\mu^{*K}\) is the \(K\)-fold convolution of \(\mu\) with itself.

This does not change the stationary measures \(\nu\) or \(\nu'\), and the fiberwise entropies and exponents are multiplied by \(K\).  Hence, lemma \ref{mainonesteplemma} yields
\[\frac{\kappa}{\chi+\epsilon} \le \underline{\dim}(\nu) \le \overline{\dim}(\nu) \le \frac{\kappa}{\chi-\epsilon}.\]

Since this holds for \(\epsilon > 0\) arbitrarily close to zero we obtain \(\dim(\nu) = \kappa/\chi\) as claimed.

\subsection{Approximation of fiberwise entropy}

\begin{lemma}\label{onestepentropyformulalemma}
  Let \(f:\Omega \to \R\) be defined by
  \[f(\omega) = \log\left(\frac{dg_{-1}(\omega)_*\nu_{\sigma^{-1}(\omega)}}{d\nu_{\omega}}\left(x(\omega)\right)\right).\]

  Then \(f\) is \(m\)-integrable and \(\int f dm = \kappa\).
\end{lemma}
\begin{proof}
  This follows immediately from lemma \ref{entropypropertieslemma}.
\end{proof}

%In order to prove lemma \ref{mainonesteplemma}, we will use the following approximation of fiberwise entropy for a suitable choice of radii \(r_n\).
In order to prove lemma \ref{mainonesteplemma}, we will approximate the function \( f\) by its averages over balls in \(\X_{T,T'}^{x'(\om)}\).

\begin{lemma}\label{makerapproximationlemma}
  For each \(r > 0\) let \(f_r:\Omega \to \R\) be defined by
  \[f_r(\omega) = \log\left(\frac{\nu_{\sigma^{-1}(\omega)}\left(g_{-1}(\omega)^{-1}B_{\omega}(x(\omega),r)\right)}{\nu_{\omega}\left(B_\omega(x(\omega),r)\right)}\right).\]

Then \(\sup\limits_{r > 0}|f_r|\) is \(m\)-integrable and \(\lim\limits_{r \downarrow 0}f_r(\omega) = f(\omega)\) for \(m\)-a.e. \(\omega \in \Omega\).
\end{lemma}
\begin{proof}
  From the Lebesgue differentiation theorem it follows directly that \(\lim\limits_{r \downarrow 0}f_r(\omega) = f(\omega)\) for \(m\)-a.e. \(\omega \in \Omega\).

  Since each \(\X_{T,T'}^{x'}\) is endowed with the distance coming from a norm on \(\R^{d_id_j}\) there is a uniform Besicovitch constant \(\beta\) valid on all fibers.

  For the lower bound we write for  \(t > 0\) 
  \[A_{t}(\omega) = \lbrace x \in \X_{T,T'}^{x'(\omega)}:  \nu_{\sigma^{-1}(\omega)}(g_{-1}(\omega)^{-1}B(x,r)) \le e^{-t}\nu_{\omega}(B(x,r))\text{ for some }r > 0\rbrace,\]
  and notice that \(A_t\) is \(x'(\omega), g_{-1}(\omega),g_{0}(\omega),g_{1}(\omega),\ldots\)-measurable.

  By the Besicovitch covering lemma we may sum over a cover of \(A_t(\omega)\) with at most \(\beta\) balls overlapping at once to obtain  
  \[\nu_{\sigma^{-1}(\omega)}(g_{-1}(\omega)^{-1}A_t(\omega)) \le \beta e^{-t},\]
  for \(m\)-a.e. \(\omega \in \Omega\).

  Since \(\nu_{\sigma^{-1}(\omega)}\) is the conditional distribution of \(x(\sigma^{-1}(\omega))\) conditioned on \(x'(\omega), g_{-1}(\omega),g_{0}(\omega),\ldots\) we obtain
  \[m\left(\lbrace \omega: \inf\limits_{r > 0}f_r(\omega) < -t\rbrace\right) = m\left(\lbrace \omega: x(\sigma^{-1}(\omega)) \in g_{-1}(\omega)^{-1}A_t(\omega)\rbrace\right) \le \beta e^{-t}.\]

  This implies that \(\inf f_r\) is bounded below by an integrable function.

  For the upper bound we consider given \(x'(\omega)\) the maximal function on the fiber \(\X_{T,T'}^{x'(\omega)}\) defined by
  \[M_{\omega}h(x) = \sup\limits_{r > 0} \frac{1}{\nu_\omega(B(x,r))}\int h d\nu_\omega,\]
  and notice that
  \[\sup\limits_{r > 0}f_r(\omega) = \log\left(M_\omega \varphi_\omega (x(\omega))\right),\]
  where
  \[\varphi_\omega = \frac{dg_{-1}(\omega)_*\nu_{\sigma^{-1}(\omega)}}{d\nu_\omega}.\]

  Since \(\varphi_\omega \nu_{\omega}\) is the conditional distribution of \(x(\omega)\) given \(x'(\sigma^{-1}(\omega)),g_{-1}(\omega),g_0(\omega),\ldots\) we obtain integrating first over this \(\sigma\)-algebra
  \begin{equation}\label{conditionalmessequation}
  \int \log\left(M_\omega \varphi_\omega (x(\omega))\right)dm(\omega) = \int\limits_{\Omega}\int\limits_{X_{T,T'}^{x'(\omega)}}\varphi_\omega(x)\log M_\omega\varphi_\omega(x) d\nu_\omega(x) dm(\omega).
\end{equation}

We now use the inequality \(a\log(b) \le a\log(a) + b/e\) to upper bound the right-hand side of equation \ref{conditionalmessequation} by
\[\int\limits_{\Omega}\int\limits_{X_{T,T'}^{x'(\omega)}}\varphi_\omega(x)\log \varphi_\omega(x) d\nu_\omega(x) dm(\omega) + \frac{1}{e}\int\limits_{\Omega}\int\limits_{X_{T,T'}^{x'(\omega)}}M_\omega\varphi_{\omega}(x) d\nu_\omega(x) dm(\omega).\]

Since \(\varphi_\omega \nu_\omega\) is the distribution of \(x(\omega)\) conditioned on \(x'(\omega),g_{-1}(\omega),g_{0}(\omega),\ldots\) the first term above is \(\kappa\).
  
The integral in the second term can be bound by \(C(1 + \int \varphi_\omega \log \varphi_\omega d\nu_\omega)\) where \(C\) depends only on the Besicovitch constant \(\beta\) and can be taken to be \(4\beta \log(2) + 4\beta\).  
This follows from \cite[lemma 9]{lessa}.

Hence, \(\sup\limits_{r > 0}|f_r|\) is \(m\)-integrable, as claimed.
\end{proof}

\begin{lemma}[Approximation of entropy]\label{onestepapproximationlemma}
  Let \(r_n: \Omega \to (0,+\infty)\) be a measurable sequence of positive functions such that 
\[\lim\limits_{n \to +\infty}r_n(\omega) = 0,\]
for \(m\)-a.e. \(\omega \in \Omega\).

Then
\[\kappa = \lim\limits_{n \to +\infty}\frac{1}{n}\sum\limits_{k = 0}^{n-1}f_{r_{n-k}(\sigma^{-k}(\omega))}(\sigma^{-k}(\omega)),\]
for \(m\)-a.e. \(\omega \in \Omega\).
\end{lemma}
\begin{proof}
  By the ergodic theorem and lemma \ref{onestepentropyformulalemma}
  \[\kappa = \lim\limits_{n \to +\infty}\frac{1}{n}\sum\limits_{k = 0}^{n-1}f\circ \sigma^{-k},\]
  at \(m\) almost every point.

  By lemma \ref{makerapproximationlemma} one has that \(\sup|f_{r_n}|\) is \(m\)-integrable and \(f_{r_n}\) converges to \(f\) almost everywhere.

  In this situation the Maker ergodic theorem allows one to substitute \(f\circ \sigma^{-k}\) for  \(f_{r_{n-k}\circ \sigma^{-k}}\circ \sigma^{-k}\) in the Birkhoff averages without altering the limit. 
\end{proof}

\subsection{Proof of the lower bound}

We will now prove the lower bound of lemma \ref{mainonesteplemma}.

For this purpose we set
\[r_n(\omega) = s_{d_id_j}(A_{-1}^0(\omega))\cdots s_{d_id_j}(A_{-n}^{-(n-1)}(\omega)).\]

By the ergodic theorem we have
\[\lim\limits_{n \to +\infty}\frac{1}{n}\log r_n(\omega) = I_{d_id_j},\]
for \(m\)-a.e. \(\omega \in \Omega\).

Since \(s_{d_id_j}(A_{-k}^{-k+1}(\omega))\) is the smallest singular value of \(A_{-k}^{-k+1}(\omega)\) we obtain
\[g_{-(k+1)}(\omega)^{-1}B_{\sigma^{-k}(\omega)}(x(\sigma^{-k}(\omega)), r_{n-k}(\sigma^{-k}(\omega))) \subset B_{\sigma^{-(k+1)}(\omega)}(x(\sigma^{-(k+1)}(\omega),r_{n-(k+1)}(\sigma^{-(k+1)}(\omega))),\]
for \(m\)-a.e. \(\omega \in \Omega\) and \(k = 0,\ldots, n-1\).

Applying this property and lemma \ref{onestepapproximationlemma} we obtain
\begin{align*}
  \kappa &= \lim\limits_{n \to +\infty}\frac{1}{n}\sum\limits_{k = 0}^{n-1}f_{r_{n-k}(\sigma^{-k}(\omega))}(\sigma^{-k}(\omega))
      \\ &\le \liminf\limits_{n \to +\infty}-\frac{1}{n}\log\left(\nu_{\omega}(B_\omega(x(\omega),r_n(\omega))\right).
\end{align*}

It follows that
\[\frac{\kappa}{I_{d_id_j}} \le \liminf\limits_{n \to +\infty}\frac{\log \nu_{\omega}(B_\omega(x(\omega),r_n(\omega))}{\log r_n(\omega)}\]
for \(m\)-a.e. \(\omega \in \Omega\), which proves the claim.

\subsection{Proof of the upper bound}

We will now prove the upper bound of lemma \ref{mainonesteplemma}.

For this purpose we set
\[r_n(\omega) = s_{1}(A_{-1}^0(\omega))\cdots s_{1}(A_{-n}^{-(n-1)}(\omega)).\]

By the ergodic theorem we have
\[\lim\limits_{n \to +\infty}\frac{1}{n}\log r_n(\omega) = I_{1},\]
for \(m\)-a.e. \(\omega \in \Omega\).

Since \(s_{1}(A_{-k}^{-k+1}(\omega))\) is the largest singular value of \(A_{-k}^{-k+1}(\omega)\) we obtain
\[g_{-(k+1)}(\omega)^{-1}B_{\sigma^{-k}(\omega)}(x(\sigma^{-k}(\omega)), r_{n-k}(\sigma^{-k}(\omega))) \supset B_{\sigma^{-(k+1)}(\omega)}(x(\sigma^{-(k+1)}(\omega),r_{n-(k+1)}(\sigma^{-(k+1)}(\omega))),\]
for \(m\)-a.e. \(\omega \in \Omega\) and \(k = 0,\ldots, n-1\).

Applying this property and lemma \ref{onestepapproximationlemma} we obtain
\begin{align*}
  \kappa &= \lim\limits_{n \to +\infty}\frac{1}{n}\sum\limits_{k = 0}^{n-1}f_{r_{n-k}(\sigma^{-k}(\omega))}(\sigma^{-k}(\omega))
      \\ &\ge \limsup\limits_{n \to +\infty}\frac{1}{n}\log\left(\frac{\nu_{\sigma^{-n}(\omega)}(B_{\sigma^{-n}(\omega)}(x(\sigma^{-n}(\omega)),1)}{\nu_{\omega}(B_\omega(x(\omega),r_n(\omega))}\right).
\end{align*}

Here, we cannot ignore the numerator \(\nu_{\sigma^{-n}(\omega)}\left(B_{\sigma^{-n}(\omega)}(X(\sigma^{-n}(\omega)),1))\right)\) which might be arbitrarily small.

However, since \(\nu_{\omega}\) is the conditional distribution of \(x(\omega)\) given \(x'(\omega),g_0(\omega),g_1(\omega),\ldots\) we may apply the Lebesgue density theorem to conclude that there exist \(c,p > 0\) such that
\[m\left(\lbrace \omega \in \Omega: \nu_\omega\left(B_\omega(x(\omega),1))\right) > c\rbrace \right) = p > 0.\]

By the ergodic theorem \(\sigma^{-n}(\omega)\) belongs to this set for \(n\) in a subsequence \(n_k(\omega)\) with density \(p\).  So we may conclude that
\[\kappa \ge \limsup\limits_{k \to +\infty}-\frac{1}{n_k(\omega)}\log\left(\nu_\omega\left(B_\omega(x(\omega),r_{n_k(\omega)}(\omega))\right)\right).\]

Since
\[\lim\limits_{k \to +\infty}\frac{\log(r_{n_k(\omega)}(\omega))}{k} = \lim\limits_{k \to +\infty}\frac{n_k(\omega)}{k}\frac{\log(r_{n_k(\omega)}(\omega))}{n_k(\omega)} = -\frac{I_1}{p} < 0,\]
it suffices to calculate the local dimension of \(\nu_\omega\) at \(x(\omega)\) along this subsequence (see \cite[remark following proposition 2.1]{young1982}).

Hence we obtain
\[\frac{\kappa}{I_{1}} \ge \lim\limits_{n \to +\infty}\frac{\log \nu_{\omega}(B_\omega(x(\omega),r_{n_k(\omega)}(\omega))}{\log r_{n_k(\omega)}(\omega)}\]
for \(m\)-a.e. \(\omega \in \Omega\).

Since this holds for \(m\)-a.e. \(\omega \in \Omega\) we obtain
\[\overline{\dim}(\nu_\omega) \le \frac{\kappa}{I_1},\]
which concludes the proof.

\section{Additivity along monotone paths}\label{final}
We assume in this section that \(\mu\) has finite first moment and countable support.

We fix a monotone path \( T^k = T \onefiner T^{k-1} \onefiner \cdots \onefiner T^0 = T'\).

\subsection{Proof of theorem \ref{additivitytheorem}}

For each \(t, t=0 \ldots , k, \) \( \nu _{T^t}\)-a.e. \(x^t \in \X_{T^t}\), the following functions are \(\nu _{T, T^t}^{x^t} \)-a.e. constant on \( \pi _{T, T^t}^{-1} (x^t) :\)
\[ \un \delta ^t (y):= \liminf _{r\to 0} \frac {\log \nu _{T, T^t}^{x^t}( B (y ,r))}{\log r}, \;\; \ov \delta^t(y) := \limsup_{r\to 0} \frac {\log \nu _{T, T^t}^{x^t} ( B (y',r))}{\log r}. \]
Moreover, the  \(\nu _{T, T^t}^{x^t} \)-a.e. constant respective  values \(\un \delta^t (x^t)\) and \(\ov \delta ^t (x^t)\) are \( \nu _{T^t}\)-a.e. constant and we denote \(\un \delta^t \) and \(\ov \delta ^t\) these respective values.
With this notation,  lower and upper dimensions of the measure \( \nu _{T,T'}^{x^{t'}} \) are respectively \(\un \delta^0 \) and \(\ov \delta ^0.\)

Theorem \ref{additivitytheorem} follows by summing the following relations (and observing that \(\un \delta^k =0  \) and \(\ov \delta ^k=0\)) for \( t = 1,\ldots , k\)
\begin{eqnarray}\label{induction1}  \un \delta^{t-1} &\geq & \un \delta^{t} +  \frac{\kappa_{T^{t}, T^{t-1}}}{\chi_{T^{t}, T^{t-1}}}\\
\label{induction2} \ov \delta^{t-1}& \leq& \ov \delta^{t} +  \frac{\kappa_{T^{t}, T^{t-1}}}{\chi_{T^{t}, T^{t-1}}}. \end{eqnarray}

The proof are similar to the proofs in \cite[section 8]{ll1}. Relation (\ref{induction1}) follows from  theorem \ref{onesteptheorem} and \cite{ledrappier-young}, Lemma 11.3.1. Given theorem \ref{onesteptheorem} and that the sequence \( \chi_{T^t, T^{t-1}}\) is nonincreasing,  the proof of (\ref{induction2}) is completely parallel to the proof of Theorem 2.6 in \cite[section 8]{ll1}, with lemma \ref{approximationlemma} replacing lemma 8.1.2.

\subsection{Proof of theorem \ref{exactdimensiontheorem}}

In order to prove theorem \ref{exactdimensiontheorem}, it remains to connect the measures \( \nu _L\) of theorem \ref{exactdimensiontheorem}. with the above conditional measures. 

Given a left filtration \(L\), we apply lemma \ref{monotone} to the pair \(T_L \prec T_0\) of admissible topologies, where \(T_0\) is the coarsest admissible topology and \(T_L \) is generated by \( T_0\) and \(L\). We obtain a monotone path 
 \[ T^k = T_L \onefiner T^{k-1} \onefiner \cdots \onefiner T^0 = T_0\]
 and the corresponding ordered of differences of exponents 
 \[\chi_{T^1,T^0} \le \chi_{T^2,T^1} \le \cdots \le \chi_{T^{k},T^{k-1}}.\]
For \( \nu _{T_0} \) -a.e. \( y\in \F\), the set \( \F_L^{y} \) has full \( \nu _L\) measure and is identified with \( \pi _{T_L,T_0}^{-1} (y)\) by a bilipschitz mapping  (see Lemma \ref{filteredvsflaglemma}). The projections \( \pi_{T, T^t} :   \pi _{T_L,T_0}^{-1} (y) \to  \pi _{T^t,T_0}^{-1} (y)\) define  a system of finer and finer projections \( \pi _t, t = 0, 1, \ldots, k \) on \( \F_L^y \). The fibers between two successive projections carry the  disintegrations of the  images of \( \nu _L .\) By  theorem  \ref{onesteptheorem}, these disintegrations are almost everywhere exact-dimensional with dimension \( \gamma _t \) given by, for \( t = 1, \ldots ,k\), 
\[ \gamma _t \, := \, \frac{ I (g_{-1}, E_{T^t} | E_{T^{t-1}}) }{\chi _{T^t, T^{t-1}}} .\]
Therefore, by theorem \ref{additivitytheorem}, the dimension of \( \nu _L\) is given by \[\dim \nu _L \; =\; \sum _{t=1}^k \gamma _t ,\]
 The geometric picture of the fibrations \( \pi _t ^{-1} \) depends on \(y \in \F\), but not the geometric numbers \( \gamma _t.\)

\end{document}